\documentclass[english]{article}
\usepackage[T1]{fontenc}
\usepackage[utf8]{inputenc}
\usepackage{geometry}
\usepackage[active]{srcltx}
\usepackage{float}

\usepackage{booktabs}
\usepackage{multicol}
\usepackage{multirow}
\usepackage{amsmath}
\usepackage{graphicx}
\usepackage{amsthm}
\usepackage{amssymb}
\usepackage{xcolor}

\usepackage{hyperref}

\makeatletter

\floatstyle{ruled}
\newfloat{algorithm}{tbp}{loa}
\providecommand{\algorithmname}{Algorithm}
\floatname{algorithm}{\protect\algorithmname}

\numberwithin{equation}{section}
\numberwithin{figure}{section}
\theoremstyle{plain}
\newtheorem{thm}{\protect\theoremname}
\theoremstyle{remark}
\newtheorem{rem}[thm]{\protect\remarkname}
\theoremstyle{plain}
\newtheorem{prop}[thm]{\protect\propositionname}

\allowdisplaybreaks
\usepackage{algorithm,algpseudocode}

\makeatother

\usepackage{babel}
\providecommand{\propositionname}{Proposition}
\providecommand{\remarkname}{Remark}
\providecommand{\theoremname}{Theorem}

\begin{document}
\global\long\def\ve{\varepsilon}%
\global\long\def\R{\mathbb{R}}%
\global\long\def\C{\mathbb{C}}%
\global\long\def\F{\mathbb{F}}%
\global\long\def\Rn{\mathbb{R}^{n}}%
\global\long\def\Rd{\mathbb{R}^{d}}%
\global\long\def\E{\mathbb{E}}%
\global\long\def\P{\mathbb{P}}%
\global\long\def\bx{\mathbf{x}}%
\global\long\def\vp{\varphi}%
\global\long\def\ra{\rightarrow}%
\global\long\def\smooth{C^{\infty}}%
\global\long\def\Tr{\mathrm{Tr}}%
\global\long\def\bra#1{\left\langle #1\right|}%
\global\long\def\ket#1{\left|#1\right\rangle }%
\global\long\def\ud{\mathrm{d}}%
\global\long\def\blam{\boldsymbol{\lambda}}%
\global\long\def\Blam{\boldsymbol{\Lambda}}%
\global\long\def\balph{\boldsymbol{\alpha}}%

\newcommand{\ML}[1]{{\color{red}[ML: #1]}}

\title{Fast entropy-regularized SDP relaxations\\for permutation synchronization}
\author{Michael Lindsey\thanks{Department of Mathematics, University of California, Berkeley; Computational Research Division, Lawrence Berkeley National Laboratory} \and Yunpeng Shi\thanks{Department of Mathematics, University of California, Davis}}
\maketitle
\begin{abstract}
We introduce fast randomized algorithms for solving semidefinite programming (SDP) relaxations of the partial permutation synchronization (PPS) problem, a core task in multi-image matching with significant relevance to 3D reconstruction. Our methods build on recent advances in entropy-regularized semidefinite programming and are tailored to the unique structure of PPS, in which the unknowns are partial permutation matrices aligning sparse and noisy pairwise correspondences across images. We prove that entropy regularization resolves optimizer non-uniqueness in standard relaxations, and we develop a randomized solver with nearly optimal scaling in the number of observed correspondences. We also develop several rounding procedures for recovering combinatorial solutions from the implicitly represented primal solution variable, maintaining cycle consistency if desired without harming computational scaling. We demonstrate that our approach achieves state-of-the-art performance on synthetic and real-world datasets in terms of speed and accuracy. Our results highlight PPS as a paradigmatic setting in which entropy-regularized SDP admits both theoretical and practical advantages over traditional low-rank or spectral techniques.
\end{abstract}

\section{Introduction}

Multi-image Matching (MIM) is a fundamental problem in computer vision, critical to numerous applications, especially in 3D reconstruction pipelines such as Structure from Motion (SfM) \cite{sfmsurvey_2017}. The goal of SfM is to reconstruct the 3D geometry of static scenes, like buildings or streets, from multiple 2D images captured by cameras with unknown locations and orientations. Each image is represented by a set of keypoints, where each keypoint ideally corresponds to a unique 3D point from a common universe. Establishing accurate and consistent matches of keypoints across multiple images, or equivalently, aligning keypoints with their correct identities from this shared universe, is essential for precise 3D reconstruction.

Mathematically, the MIM task can be formulated as a \emph{partial permutation synchronization (PPS)} problem \cite{chen_partial, MatchFAME, MatchALS}. In PPS, each image is modeled as a node within a viewing graph. Both pairwise (``relative'') keypoint matches between images and the global (``absolute'') keypoint-to-universe correspondences are represented as partial permutation matrices. Recall that a partial permutation matrix is a binary matrix with at most one nonzero entry per row and per column. Ideally, the relative match between images $i$ and $j$, denoted by $Q^{(i,j)}$, should satisfy:
\[
Q^{(i,j)} = P^{(i)} P^{(j)\top},
\]
where $P^{(i)}$ is the absolute partial permutation corresponding to image $i$. Given noisy and partially observed measurements $\{Q^{(i,j)}\}$, the objective of PPS is to recover the absolute partial permutations $\{P^{(i)}\}$. A precise formulation of the PPS problem is provided in Section \ref{sec:prelim}.

Although MIM is crucial for accurate 3D reconstruction, it poses several computational and practical challenges. In SfM, keypoints are typically detected using local feature detectors such as SIFT \cite{sift04} or deep learning-based methods \cite{deepmatch_survey}. However, due to viewpoint variations, occlusions, and measurement noise, initial pairwise keypoint matches are often incomplete and corrupted. Recovering ground truth correspondences from such imperfect data is at the core of the PPS challenge.

Practically, solving PPS involves determining the set of absolute partial permutation matrices associated with nodes in the viewing graph from potentially corrupted relative measurements encoded on the graph edges. When the partial permutation matrices are full (square, binary, and bistochastic), PPS simplifies to permutation synchronization (PS) \cite{deepti}, a special instance of the more general group synchronization problem. For instance, rotation averaging problems in $SO(3)$ \cite{HartleyAT11_rotation, ChatterjeeG13_rotation, MPLS} are typically addressed by spectral methods \cite{singer2011angular} or semidefinite programming (SDP) relaxations \cite{wang_singer_2013}. However, in image matching, the assumption of uniform keypoint visibility across images rarely holds, making PPS a more realistic and general framework. The PPS problem is notably challenging due to:

\begin{itemize}
    \item \textbf{High dimensionality}: Partial permutation matrices associated with each image can have hundreds or thousands of rows and columns, significantly larger than typical rotation averaging dimensions (usually three). Standard SDP methods scale poorly in such scenarios, presenting severe computational hurdles.

    \item \textbf{Heterogeneous sizes and sparsity}: Unlike the uniform square matrices in PS, relative partial permutation matrices in PPS often vary in size and sparsity. This heterogeneity introduces biases and numerical instabilities into conventional synchronization algorithms.
\end{itemize}

\subsection{Related Work}

Several approaches have addressed the PPS and PS problems. Spectral methods \cite{deepti} compute eigenvectors of block matrices constructed from relative permutations, followed by projection onto permutation matrices using algorithms such as the Hungarian method \cite{Munkres}. Extensions to PPS involve heuristic projection techniques \cite{ConsistentFeature, inv_semi_group}. Methods like MatchEIG \cite{MatchEIG} improve computational efficiency but may yield overly sparse solutions.

SDP-based approaches like MatchLift \cite{chen_partial} handle uniform corruption scenarios but suffer from high computational and memory costs, limiting their scalability. MatchALS \cite{MatchALS} introduces linear constraints to improve computational efficiency, yet still faces limitations in large-scale or real-world scenarios.

Memory efficiency is another critical limitation for SDP-based methods, which typically require storing large dense matrices. Recent nonconvex methods, such as the projected power method (PPM) \cite{Chen_PPM}, use blockwise iterations and offer significant memory reduction, though they are sensitive to initialization. Cycle-edge message passing (CEMP) \cite{cemp} avoids spectral initialization and offers robust, memory-efficient synchronization, with MatchFAME \cite{MatchFAME} extending CEMP to PPS scenario. FCC \cite{FCC} introduces confidence scoring for keypoints but does not guarantee cycle consistency.

In contrast, our approach substantially reduced time/space complexity and allows enforcement of cycle consistency. 

\subsection{Our Contributions}

We propose a fast randomized SDP approach specifically tailored for solving PPS. Our approach leverages randomization and entropic regularization to reduce both computational complexity and memory usage while ensuring robustness to nonuniform corruption and noise. The primary contributions of this work are as follows:

\begin{itemize}
    \item We introduce novel randomized algorithms for an entropy-regularized SDP relaxation of PPS, building on \cite{lindsey2023fastrandomizedentropicallyregularized} and achieving scalability for large-scale SfM datasets.

    \item We develop fast rounding procedures for recovering combinatorial solutions from the implicitly represented primal solution variable of the regularized SDP. These rounding procedures avoid the introduction of a separate computational bottleneck and can maintain cycle consistency in the recovered solution if desired.

    \item We provide theoretical justification for the effectiveness of entropic regularization in reducing computational complexity and enhancing solution accuracy.

    \item Extensive experiments on real-world datasets demonstrate state-of-the-art performance of our method in terms of speed, accuracy, and memory efficiency, outperforming previous methods such as MatchEIG \cite{MatchEIG} and MatchFAME \cite{MatchFAME}.
\end{itemize}

We also aim to highlight the PPS problem as a prototypical problem where randomized entropy-regularized semidefinite programming enjoys an asymptotic computational scaling advantage, relative to both low-rank SDP solvers~\cite{burer2003nonlinear,manopt,SketchyCGAL} and spectral methods. In order to explain this advantage we must introduce some further background. Note that our discussion assumes a constant or only logarithmically growing number of iterations for convergence to a fixed tolerance, so our comparison of computational complexity considers only the per-iteration costs. Such an assumption holds empirically in our experiments as well as the previous work \cite{lindsey2023fastrandomizedentropicallyregularized}. For more detailed theoretical analysis, we point to a companion paper by one of the authors~\cite{cailindsey2025}, in which we rigorously analyze the convergence theory of entropy-regularized SDP solvers, as well as the effect of regularization on solution accuracy.

Now the primal semidefinite variable $X$ in this application is an $NK \times NK$ matrix, where $N$ denotes the number of images and $K$ the number of keypoints per image (assumed constant in this discussion for simplicity). The rank of the primal variable induced by the exact combinatorial solution is $M$, the total number of unique keypoints in the universe. Meanwhile, we also form a sparse $NK \times NK$ cost matrix $Q$ which encodes the observed (possibly corrupted) correspondences that we are given. The number of these correspondences is $\mathrm{nnz} (Q)$. Note that in the case of full PS, we have $M=K$ and $\mathrm{nnz} (Q) = N^2 K$. Then in general, we view $\tilde{O}( \mathrm{nnz}(Q) )$ computational scaling as near-optimal, where the tilde conceals logarithmic factors.

An iteration of a low-rank SDP or spectral solver typically requires at least $O(M)$ matrix-vector multiplications by the cost matrix, as well as orthogonalization of a collection of $O(M)$ vectors of size $NK$, yielding a total computational complexity of at least $O( \mathrm{nnz}(Q) M + NK M^2)$ per iteration. In addition, there is a memory cost of $O(NKM)$ to represent the low-rank solution. For SDP solvers, there may be additional costs based on the structure of the constraints.

The motivation of entropy-regularized semidefinite programming \cite{lindsey2023fastrandomizedentropicallyregularized} is precisely to avoid dependence on the solution rank $O(M)$, which in this context is not expected to be small. We present two entropy-regularized SDP relaxations together with corresponding algorithms, which we refer to as the strong and weak SDP. For the strong SDP, which is a tighter relaxation, we achieve $\tilde{O} ( \mathrm{nnz}(Q) K + NK^3) $ computational cost, and the memory cost is $O(\mathrm{nnz}(Q) + NK^2)$. For the weak SDP, we achieve the optimal scaling $\tilde{O} ( \mathrm{nnz}(Q) )$ for both cost and memory. Note that in either case, there is no explicit dependence on the solution rank.

We also develop two rounding procedures for recovering a cycle-consistent collection of partial permutations from the implicitly represented primal variable $X$, which we can interact with efficiently only via matrix-vector multiplications. These `slow' and `fast' recovery procedures match the computational scalings of the strong and weak SDP solvers, respectively. In our practical experiments, we observe that the results furnished by the weak SDP with fast recovery are very similar to those furnished by the strong SDP with slow recovery, justifying the use of these faster algorithms in practice.

Finally we present a `masked' recovery procedure, which does not necessarily furnish a cycle-consistent solution, but instead classifies the observed correspondences as correct or incorrect. The cost of this simpler procedure is $\tilde{O} ( \mathrm{nnz} (Q) )$.

To conclude, we comment on a key advantage of our approaches over spectral methods, namely that the solution rank, i.e., the total number of unique keypoints, need not be specified \emph{a priori}, but can rather be recovered from the data via the recovery procedure.

\subsection{Outline}
In Section~\ref{sec:prelim} we provide background on the full and partial permutation synchronization problems. In Section~\ref{sec:entropic} we prove that entropy regularization disambiguates the degeneracy of standard SDP relaxations of the PPS in order to promote a single minimal solution. In Section~\ref{sec:solvers} we introduce fast randomized solvers for the strong and weak regularized SDP relaxations. In Section~\ref{sec:recovery} we present the `slow,' `fast,' and `masked' rounding procedures furnishing partial permutations from the dual solution of the regularized SDP. In Section~\ref{sec:experiments}, we present numerical experiments evidencing the effectiveness of our algorithms.

\subsection{Acknowledgments}
This material is based on work supported by the U.S. Department of
Energy, Office of Science, Accelerated Research in Quantum Computing
Centers, Quantum Utility through Advanced Computational Quantum Algorithms (M.L.),
grant no. DE-SC0025572 and by the Applied Mathematics Program of the
US Department of Energy (DOE) Office of Advanced Scientific Computing
Research under contract number DE-AC02-05CH11231 (M.L.). M.L. was also partially
supported by a Sloan Research Fellowship. Y.S. is supported by the startup fund from UC Davis Department of Mathematics.

\section{Preliminaries}\label{sec:prelim}

In this section, we first explain how SDP relaxations are formally
derived from the full permutation synchronization problem. We also
highlight the difficulty of identifying the right relaxation of the
partial permutation synchronization problem.

\subsection{Full permutation synchronization \label{sec:full}}

In the setting of full permutation synchronization, we suppose that
we are given permutation matrices $Q^{(i,j)}\in\R^{K\times K}$ for
$i,j=1,\ldots,N$, each encoding a permutation of $K$ elements. By
assumption $Q^{(i,i)}=\mathbf{I}_{K}$ and $Q^{(i,j)\top}=Q^{(j,i)}$
for all $i,j=1,\ldots,N$. The goal, roughly speaking, is to determine
permutations $P^{(i)},\ldots,P^{(N)}$, where we assume without loss
of generality that $P^{(1)}=\mathbf{I}_{K}$, such that $Q^{(i,j)}\approx P^{(i)}P^{(j)\top}$
for all $i,j=1,\ldots,N$.

Since this may not be possible to achieve exactly, one motivation
for typical semidefinite relaxations is to minimize the objective
\begin{equation}
\sum_{i,j=1}^{N}\Vert Q^{(i,j)}-P^{(i)}P^{(j)\top}\Vert_{\mathrm{F}}^{2}=\Vert Q-PP^{\top}\Vert_{\mathrm{F}}^{2}\label{eq:frobloss}
\end{equation}
 over all permutation matrices $P^{(2)},\ldots,P^{(N)}$, where we
fix $P^{(1)}=\mathbf{I}_{K}$, $Q\in\R^{NK\times NK}$ denotes the
matrix with suitable blocks $Q^{(i,j)}$, and we implicitly define
$P$ in terms of the $P^{(i)}$ via 
\begin{equation}
P=\left(\begin{array}{c}
P^{(1)}\\
P^{(2)}\\
\vdots\\
P^{(N)}
\end{array}\right).\label{eq:Pblock}
\end{equation}
Since $\Vert Q^{(i,j)}\Vert_{\mathrm{F}}^{2}=\Vert P^{(i)}P^{(j)\top}\Vert_{\mathrm{F}}^{2}=K$
for all $i,j$, it is equivalent to maximize the objective
\[
\Tr[QPP^{\top}]
\]
 over all $P$ of the form (\ref{eq:Pblock}) in which the blocks
$P^{(i)}\in\R^{K\times K}$ are permutation matrices.

We may relax this problem by replacing $PP^{\top}$ with $X\in\R^{NK\times NK}$,
enforcing certain conditions on $X$ that are necessarily satisfied
by $PP^{\top}$, specifically the conditions that $X\succeq0$ and
that the diagonal blocks of $X$ are known. The resulting semidefinite
relaxation reads as: 
\begin{align}
\underset{X\in\R^{NK\times NK}}{\text{maximize}}\ \  & \Tr[QX]\label{eq:strongSDP}\\
\text{subject to}\ \  & X^{(i,i)}=\mathbf{I}_{K},\quad i=1,\ldots,N,\nonumber \\
 & X\succeq0.\nonumber 
\end{align}
 In the first constraint, and henceforth, we use the superscript $(i,j)$
to denote suitable blocks. We will refer to this semidefinite relaxation
as the `strong SDP' in the context of this work.

We will also consider the following `weak SDP' relaxation, which can
be derived by further relaxing the constraints of (\ref{eq:strongSDP}):
\begin{align}
\underset{X\in\R^{NK\times NK}}{\text{maximize}}\ \  & \Tr[QX]\label{eq:weakSDP}\\
\text{subject to}\ \  & \mathrm{diag}(X)=\mathbf{1}_{NK},\nonumber \\
 & \Tr\left[X^{(i,i)}\frac{\mathbf{1}_{K}\mathbf{1}_{K}^{\top}}{K}\right]=1,\quad i=1,\ldots,N,\nonumber \\
 & X\succeq0.\nonumber 
\end{align}
 Here and below, we use the notation $\mathbf{1}_{n}\in\R^{n}$ to
denote the vector of all 1's. Note that the second constraint specifies
that the sum of all entries in each diagonal block $X^{(i,i)}$ is
$1$.

\subsection{Partial permutation synchronization \label{sec:partial}}

We will now discuss the more general and difficult case of partial
permutation synchronization.

In the general case, we now consider $N$ blocks of potentially different
sizes $K^{(1)},\ldots,K^{(N)}$, and we are given $Q^{(i,j)}\in\R^{K^{(i)}\times K^{(j)}}$
for $i,j=1,\ldots,N$, each of which is a \emph{partial permutation
matrix}, i.e., a principal submatrix of some permutation matrix. By
assumption, as above, we have $Q^{(i,i)}=\mathbf{I}_{K^{(i)}}$ and
$Q^{(i,j)\top}=Q^{(j,i)}$ for all $i,j$. It is useful to define
$L:=\sum_{i=1}^{N}K^{(i)}$, which will correspond to the size of
our large matrices.

The problem is to determine $M$, as well as some partial permutation
matrices $P^{(i)}\in\R^{K^{(i)}\times M}$, which satisfy $Q^{(i,j)}\approx P^{(i)}P^{(j)\top}$
as well as possible. In fact each $P^{(i)}$ should have exactly one
$1$ in each of its rows, i.e., $P^{(i)}$ should consist of a subset
of \emph{rows} of a full permutation matrix. We shall call such matrices
\emph{row-partial permutation matrices}. The setting of full permutation
synchronization, described above in Section \ref{sec:full}, is recovered
in the case where $K^{(i)}=K=M$ for all $i$.

It is useful to understand the problem intuitively in the language
of the important application of image registration. In this setting,
$N$ is the number of images. Within the $i$-th image there are $K^{(i)}$
pre-specified `keypoints.' The set of keypoints in each image can
be viewed as a subset of an underlying collection of `registry points'
(of size $M\geq K^{(i)}$). The partial permutations $Q^{(i,j)}$
indicate identifications, potentially corrupted, between keypoints
in the $i$-th image and keypoints in the $j$-th image. The goal
is to recover the size of the registry and to identify the correspondence
of the keypoints in each image to the registry points.

To approach this problem, we might try once again to begin with the
minimization of the same loss function (\ref{eq:frobloss}) over all
row-partial permutation matrices $P^{(i)}\in\R^{K^{(i)}\times M}$.

Unfortunately, if we attempt to follow the same derivation as above,
we encounter the difficulty that the quantity $\Vert P^{(i)}P^{(j)\top}\Vert_{\mathrm{F}}^{2}$,
which we previously discarded as an irrelevant constant, is not independent
of the choice of $P^{(1)},\ldots,P^{(N)}.$ 

As a first attempt to address this issue, one can observe that $\Vert P^{(i)}P^{(j)\top}\Vert_{\mathrm{F}}^{2}=\mathbf{1}_{K^{(i)}}^{\top}P^{(i)}P^{(j)\top}\mathbf{1}_{K^{(j)}}$,
because the Frobenius norm of a matrix whose entries are all in $\{0,1\}$
is the same as the sum of the entries. This identity suggests that
the inclusion of the additional linear term $\frac{1}{2}\mathbf{1}_{L}^{\top}X\mathbf{1}_{L}$
in the SDP objective, relative to (\ref{eq:strongSDP}), could account
for the missing term in the derivation.

However, such a penalty term ruins the interpretation of (\ref{eq:strongSDP})
of (\ref{eq:weakSDP}) as maximizing the `overlap' of $X$ with a
fixed target matrix $Q$, and in particular it encourages the primal
variable $X$ to assume negative entries. Without explicit enforcement
of the entrywise constraint $X\geq0$, this term undermines the quality
of the relaxation. Meanwhile, in practice, the enforcement of such
a dense inequality constraint on the PSD matrix $X$ is intractable
for large $L$.

Optimistically, we seek an optimization approach with optimal scaling
$O(\mathrm{nnz}(Q))=O(N^{2}K)$, where $K:=\max_{i=1,\ldots,N}K^{(i)}$,
and it will turn out, quite fortuitously, that our strategy for achieving
this scaling (namely, entropic regularization) \emph{also }addresses
the difficulty posed by the ambiguity of the derivation of an SDP
relaxation for the partial permutation problem.

To understand this ambiguity, let us examine what happens when we
simply consider SDP relaxations that are formally analogous to the
ones derived above in Section \ref{sec:full}. To wit, define the
matrix $Q=\left[Q^{(i,j)}\right]$ blockwise as above (with the caveat
that the blocks are not necessarily all of uniform size). Then our
strong SDP reads as: 
\begin{align}
\underset{X\in\R^{L\times L}}{\text{maximize}}\ \  & \Tr[QX]\label{eq:strongSDPgeneral}\\
\text{subject to}\ \  & X^{(i,i)}=\mathbf{I}_{K^{(i)}},\quad i=1,\ldots,N,\nonumber \\
 & X\succeq0,\nonumber 
\end{align}
 and our weak SDP as: 

\begin{align}
\underset{X\in\R^{L\times L}}{\text{maximize}}\ \  & \Tr[QX]\label{eq:weakSDPgeneral}\\
\text{subject to}\ \  & \mathrm{diag}(X)=\mathbf{1}_{L},\nonumber \\
 & \Tr\left[X^{(i,i)}\frac{\mathbf{1}_{K^{(i)}}\mathbf{1}_{K^{(i)}}^{\top}}{K^{(i)}}\right]=1,\quad i=1,\ldots,N,\nonumber \\
 & X\succeq0.\nonumber 
\end{align}
Again we use the superscript $(i,j)$ to denote suitable blocks, now
possibly of variable size.

The problem with both (\ref{eq:strongSDPgeneral}) and (\ref{eq:weakSDPgeneral})
can be understood in the case of uncorrupted data, i.e., where $Q^{(i,j)}=P^{(i)}P^{(j)\top}$
for ground truth row-partial permutations $P^{(i)}\in\R^{K^{(i)}\times M}$,
or writing 
\[
P=\left(\begin{array}{c}
P^{(1)}\\
\vdots\\
P^{(N)}
\end{array}\right)\in\R^{L\times M},
\]
 where 
\[
Q=PP^{\top}.
\]
 Note that for each such $P$, the $m$-th column of $P$ can be viewed
as a multi-hot encoding of the keypoints in each of the $N$ subsets
corresponding to the $m$-th registry point. Conversely, any $P\in\{0,1\}^{L\times M}$
defines a valid ground truth as long as:
\begin{enumerate}
\item each row contains exactly one 1, and 
\item each column contains at most a single 1 within each of its blocks.
\end{enumerate}
In this setting, the solutions of (\ref{eq:strongSDPgeneral}) and
(\ref{eq:weakSDPgeneral}) are \emph{not unique}. Besides the ground
truth solution specified by $X=PP^{\top}$ (or equivalently, $X^{(i,j)}=P^{(i)}P^{(j)\top}$),
there may be many alternative solutions of the form $\tilde{X}^{(i,j)}=\tilde{P}^{(i)}\tilde{P}^{(j)\top}$,
where $\tilde{P}^{(i)}\in\R^{K^{(i)}\times\tilde{M}}$ with $\tilde{M}<M$.
In the language of image registration, such a solution can be obtained
by `identifying two registry points as one, i.e., by replacing two
of the columns of $P$ with a single column which is their sum. As
long as those two registry points never appear as keypoints within
the same image, then properties (1) and (2) above will be satisfied
for the resulting $\tilde{P}$. Moreover, the nonzero pattern of $\tilde{X}=\tilde{P}\tilde{P}^{\top}\in\{0,1\}^{L\times L}$
will contain that of $PP^{\top}\in\{0,1\}^{L\times L}$, i.e., that
of $Q$, and we will have $\Tr[Q\tilde{X}]=\Tr[QX]$. Evidently, many
such $\tilde{P}$ can be constructed by merging subsets of points
that never appear in the same image, and moreover, any convex combinations
of the resulting discrete set of $\tilde{X}$ will also share the
same value while remaining feasible for (\ref{eq:strongSDPgeneral})
and (\ref{eq:weakSDPgeneral}).

We will address this difficulty in the next section.

Before proceeding, it is instructive to define an even weaker Goemans-Williamson-type~\cite{Goemans_Williamson_1995} (GW-type, for short) SDP relaxation, which removes all
constraints on the diagonal blocks of $X$ besides the constraint
that the diagonal entries themselves are all $1$:
\begin{align}
\underset{X\in\R^{L\times L}}{\text{maximize}}\ \  & \Tr[QX]\label{eq:weakestSDP}\\
\text{subject to}\ \  & \mathrm{diag}(X)=\mathbf{1}_{L},\nonumber \\
 & X\succeq0.\nonumber 
\end{align}
 Omitting these additional constraints can be viewed as completely
forgetting the fact that distinct keypoints in the same image cannot
correspond to the same registry point. Note that in particular, the
trivial solution $X=\mathbf{1}_{L}\mathbf{1}_{L}^{\top}$ is feasible
(and in fact optimal); this solution corresponds to the identification
of \emph{all} keypoints in all images with a single registry point
($M=1$).

All of the above relaxations coincide in the case where $K^{(i)}=1$
for all $i$. This setting describes the problem of registering an
unstructured collection of points, given only some (potentially corrupted)
identifications between pairs of points.

\section{The blessing of entropic regularization}\label{sec:entropic}

We saw in Section \ref{sec:partial} that, in addition to the desired
ground truth solution, any solution of (\ref{eq:strongSDPgeneral})
which `hallucinates' additional identifications between keypoints
in distinct images (beyond those indicated by the data $Q^{(i,j)}$)
will also be optimal.

In this section, we will explain how entropic regularization resolves
this ambiguity in favor of the ground truth solution, which, among
all these options, maximizes the number of points in the recovered
registry. In the next section (i.e., Section \ref{sec:solvers} below),
we will show how this regularization can \emph{also }be used to derive
fast optimizers. Our solver will represent our primal solution only
implicitly, so later, in Section \ref{sec:recovery}, we must explain
how we can efficiently query our primal solution to recover the partial
permutations that connect pairs of images, as well as the correspondences
of keypoints to an underlying registry that we construct.

Following~\cite{lindsey2023fastrandomizedentropicallyregularized}, we define the von Neumann entropy 
\[
S(X):=\Tr[X\log X]-\Tr[X]
\]
 on the domain of positive semidefinite matrices and consider the
following entropic regularization of (\ref{eq:strongSDPgeneral}): 

\begin{align}
\underset{X\in\R^{L\times L}}{\text{minimize}}\ \  & \Tr[CX]+\beta^{-1}S(X)\label{eq:strongSDPreg}\\
\text{subject to}\ \  & X^{(i,i)}=\mathbf{I}_{K^{(i)}},\quad i=1,\ldots,N,\nonumber \\
 & X\succeq0.\nonumber 
\end{align}
 Here $\beta>0$ is a regularization parameter, which can be interpreted
as an inverse temperature in the sense of quantum statistical mechanics,
and we define the cost matrix $C:=-Q$ to yield a minimization problem,
for consistency with~\cite{lindsey2023fastrandomizedentropicallyregularized}. When $\beta$ is very large, the
influence of regularization is reduced, but the optimization problem
is more difficult to solve.

Consider the case of uncorrupted data, i.e., the case where 
\begin{equation}
C=-PP^{\top},\quad P=\left(\begin{array}{c}
P^{(1)}\\
\vdots\\
P^{(N)}
\end{array}\right),\label{eq:uncorrupted}
\end{equation}
 in which the blocks $P^{(i)}\in\R^{K^{(i)}\times M}$ of $P$ are
row-partial permutation matrices. Remarkably, Theorems \ref{thm:reg}
and \ref{thm:reg2} suggest that in this case, the entropic regularization
resolves the ambiguity of the linear cost term $\Tr[CX]$, explained
in Section \ref{sec:partial}, in favor of the ground truth solution.
Concretely, as the regularization is removed (i.e., in the limit $\beta\ra\infty$),
the optimizer of (\ref{eq:strongSDP}) converges to the ground truth
$PP^{\top}$.

In fact, Theorems \ref{thm:reg} and \ref{thm:reg2} also apply to
the following entropic regularization of the weak SDP (\ref{eq:weakSDPgeneral}): 

\begin{align}
\underset{X\in\R^{L\times L}}{\text{minimize}}\ \  & \Tr[CX]+\beta^{-1}S(X)\label{eq:weakSDPreg}\\
\text{subject to}\ \  & \mathrm{diag}(X)=\mathbf{1}_{L},\nonumber \\
 & \Tr\left[X^{(i,i)}\frac{\mathbf{1}_{K^{(i)}}\mathbf{1}_{K^{(i)}}^{\top}}{K^{(i)}}\right]=1,\quad i=1,\ldots,N,\nonumber \\
 & X\succeq0,\nonumber 
\end{align}
 as well as the entropic regularization of the GW-type SDP (\ref{eq:weakestSDP}):
\begin{align}
\underset{X\in\R^{L\times L}}{\text{minimize}}\ \  & \Tr[CX]+\beta^{-1}S(X)\label{eq:weakestSDPreg}\\
\text{subject to}\ \  & \mathrm{diag}(X)=\mathbf{1}_{L},\nonumber \\
 & X\succeq0.\nonumber 
\end{align}

In our optimization framework, to be detailed in Section \ref{sec:solvers}
below, this last problem (\ref{eq:weakestSDPreg}) is no easier to
solve than (\ref{eq:weakSDPreg}), which already enjoys the near-optimal
practical scaling of $\tilde{O}(\text{nnz}(C))=\tilde{O}(N^{2}K)$.
(Here the tilde in $\tilde{O}$ indicates the omission of log factors.)
Meanwhile, the practical scaling for the strong regularized SDP (\ref{eq:strongSDPreg})
is $\tilde{O}(\text{nnz}(C)K+NK^{3})$. When this cost is managable,
we recommend (\ref{eq:strongSDPreg}). When this cost is not manageable,
we recommend the weak regularized SDP (\ref{eq:weakSDPreg}). The
GW-type regularized SDP (\ref{eq:weakestSDPreg}) is presented mostly
as a theoretical toy to emphasize the minimal sufficient condition
for exact recovery by entropic regularization. (Meanwhile, see~\cite{lindsey2023fastrandomizedentropicallyregularized}
for a numerical approach in this simplest case.) Since the scaling
is no better than that of (\ref{eq:weakSDPreg}), we see no reason
to recommend it in any case.
\begin{thm}
\label{thm:reg}Assume that the cost matrix $C=-Q$ is defined in
terms of uncorrupted data $Q=PP^{\top}$, following (\ref{eq:uncorrupted}),
in which the $P^{(i)}$ are row-partial permutations. Then $X^{\star}:=Q$
is an optimal solution of the unregularized SDPs (\ref{eq:strongSDP}),
(\ref{eq:weakSDP}), and (\ref{eq:weakestSDP}). Moreover, among all
optimal solutions $X$ for any of these problems, $X^{\star}$ minimizes
the von Neumann entropy $S(X)$.
\end{thm}

\begin{thm}
\label{thm:reg2}Assume that the cost matrix $C=-Q$ is defined in
terms of uncorrupted data $Q=PP^{\top}$, following (\ref{eq:uncorrupted}),
in which the $P^{(i)}$ are row-partial permutations. Let $X_{\beta}$
be the unique solution of the regularized SDP (\ref{eq:strongSDPreg})
with inverse temperature $\beta>0$. Then $\lim_{\beta\ra\infty}X_{\beta}=Q$.
The same result holds for the regularized SDPs (\ref{eq:weakSDPreg})
and (\ref{eq:weakestSDPreg}).
\end{thm}

\begin{rem}
In fact we can compute $X_{\beta}$ explicitly, and it is exactly
the same in each of the cases (\ref{eq:strongSDPreg}), (\ref{eq:weakSDPreg}),
and (\ref{eq:weakestSDPreg}). We shall elaborate forthwith on the
structure of the solution, but see the full proof for the detailed
construction.
\end{rem}

The proofs of Theorems \ref{thm:reg} and \ref{thm:reg2} are given
in Appendices \ref{sec:thm1} and \ref{sec:thm2}, respectively. The
main idea in the proof of Theorem \ref{thm:reg} is that $Q$ can
be viewed as a block-diagonal matrix, if we reorder the $L=\sum_{i=1}^{N}K^{(i)}$
row/column indices of $Q$ indices, by registry point correspondence,
into $M$ blocks. Moreover, the entries of each diagonal block are
all 1's. It can be shown then that maximization of $\Tr[QX]$, subject
to $X\succeq0$ and $\mathrm{diag}(X)=\mathbf{1}_{L}$, is equivalent
to agreement with $Q$ on all diagonal blocks, after reindexing. The
Gibbs variational principle~\cite{Dissertation} then establishes that the entropy-optimal
$X$ satisfying this block-diagonal constraint is itself block-diagonal.

We also sketch the proof of Theorem \ref{thm:reg2}, which makes use
of the same reindexing that block-diagonalizes $Q$. In fact, the
Gibbs variational principle can be used to show that each optimizer
$X_{\beta}$ must be block-diagonal after this reindexing. This allows
us to explicitly solve for the reindexed diagonal blocks, which are
each of the form
\[
\tau_{\beta}\mathbf{I}+(1-\tau_{\beta})\mathbf{1}\mathbf{1}^{\top},
\]
 where $\tau_{\beta}\in(0,1)$, i.e., a convex combination of the
identity and the matrix of all ones (of suitable sizes). The coefficient
$\tau_{\beta}$ can be computed explicitly and satisfies $\tau_{\beta}\ra0$
as $\beta\ra\infty$, and this implies that $X_{\beta}\ra Q$ as $\beta\ra\infty$.

\section{Fast randomized solvers \label{sec:solvers}}

In this section we will explain how we can solve our SDP relaxations
efficiently by adapting the fast randomized entropically regularized
approach of~\cite{lindsey2023fastrandomizedentropicallyregularized}. Since our SDP relaxations for the partial
permutation synchronization problems generalize our relaxations for
the full synchronization problem, we will consider the more general
setting, as outlined in Section \ref{sec:partial}. We will outline
algorithms for solving the problems (\ref{eq:strongSDPreg}) and (\ref{eq:weakSDPreg}),
the strong and weak regularized SDP, respectively.

\subsection{Strong regularized SDP}

The dual problem to (\ref{eq:strongSDPreg}) can be derived as 
\[
\underset{\Lambda^{(i)}\in\R^{K^{(i)}\times K^{(i)}},\ i=1,\ldots,N}{\text{maximize}}\ \ \mathcal{F}[\Lambda^{(1)},\ldots,\Lambda^{(N)}],
\]
 where 
\begin{equation}
\mathcal{F}[\Lambda^{(1)},\ldots,\Lambda^{(N)}]:=\sum_{i=1}^{N}\Tr[\Lambda^{(i)}]-\beta^{-1}\Tr\left[e^{-\beta\left(C-\bigoplus_{i=1}^{N}\Lambda^{(i)}\right)}\right].\label{eq:strongDualObj}
\end{equation}
 Here the direct sum $\bigoplus$ refers to the construction of a
block-diagonal matrix from the summands. The primal solution $X$
can be recovered from the dual solution as 
\begin{equation}
X_{\beta,\boldsymbol{\Lambda}}:=e^{-\beta\left(C-\bigoplus_{i=1}^{N}\Lambda^{(i)}\right)}.\label{eq:XbetaLambda}
\end{equation}
 For notational brevity, we will collect the dual variables as 
\[
\boldsymbol{\Lambda}:=[\Lambda^{(1)},\ldots,\Lambda^{(N)}]
\]
 and identify 
\[
\mathcal{F}[\boldsymbol{\Lambda}]=\mathcal{F}[\Lambda^{(1)},\ldots,\Lambda^{(N)}].
\]

\subsubsection{Optimizer}

In order to maximize the dual objective (\ref{eq:strongDualObj}),
we pursue a generalization of the `noncommutative matrix scaling'
approach advanced in~\cite{lindsey2023fastrandomizedentropicallyregularized} (for Goemans-Williamson-type relaxations
with diagonal constraint) to the case of a block-diagonal constraint.
This generalization can be derived by generalizing the derivation
of the matrix scaling approach as a minorization-based optimizer.

To-wit, suppose that $\Lambda_{0}^{(i)}$, $i=1,\ldots,N$, form our
current guess for the dual solution, and we want to solve for the
update increment $\Delta\Lambda^{(i)}:=\Lambda^{(i)}-\Lambda_{0}^{(i)}$,
$i=1,\ldots,N$. It is convenient to define 
\begin{equation}
X_{0}:=e^{-\beta\left(C-\bigoplus_{i=1}^{N}\Lambda_{0}^{(i)}\right)},\label{eq:X0}
\end{equation}
 the primal guess associated to our dual guess.

We will update the dual variables by the dual update rule:
\begin{equation}
\Delta\Lambda^{(i)}=-\beta^{-1}\log\left(X_{0}^{(i,i)}\right),\quad i=1,\ldots,N.\label{eq:dualupdate}
\end{equation}
 Evidently $\Delta\Lambda^{(i)}=0$ for all $i$ if and only if $[X_{0}]^{(i,i)}=I_{K^{(i)}}$
for all $i$, i.e., if and only if the dual optimality condition holds.
Therefore the update rule can be viewed as a legitimate fixed-point
iteration for the dual variables, with unique fixed point corresponding
to the unique optimizer. However, we can interpret this fixed-point
iteration from an optimization point of view as a `Jacobi-style' update
for a concave minorization of the dual objective, as we shall describe
in the next section.

Though the point of view of solving for the increment is useful for
the interpretation of the approach, we remark before moving on that
the fixed-point iteration can more straightforwardly be presented
as 
\begin{align}
 & X_{0}\leftarrow X_{\beta,\boldsymbol{\Lambda}}\label{eq:dualupdate2}\\
 & \Lambda^{(i)}\leftarrow\Lambda^{(i)}-\beta^{-1}\log\left(X_{0}^{(i,i)}\right),\quad i=1,\ldots,N.\nonumber 
\end{align}
 Later on (cf. Algorithm \ref{alg:strong} below), we will present
the algorithm more formally and include a slight extension allowing
for damping of the step size.

\subsubsection{Interpretation via minorization \label{sec:minorization}}

For each $i$, let $\{u_{k}^{(i)}\}_{k=1}^{K^{(i)}}$ denote an orthonormal
collection of eigenvectors for $X_{0}^{(i,i)}$. Then consider the
proxy objective: 
\begin{equation}
\mathcal{F}_{0}[\boldsymbol{\alpha}]=\sum_{i=1}^{N}\Tr[\Lambda_{0}^{(i)}]+\sum_{i=1}^{N}\Tr\left[\sum_{k=1}^{K^{(i)}}\alpha_{k}^{(i)}u_{k}^{(i)}u_{k}^{(i)\top}\right]-\beta^{-1}\sum_{i=1}^{N}\Tr\left[e^{\beta\sum_{k=1}^{K^{(i)}}\alpha_{k}^{(i)}u_{k}^{(i)}u_{k}^{(i)\top}}X_{0}^{(i,i)}\right],\label{eq:Fproxy}
\end{equation}
 where $\boldsymbol{\alpha}=[\alpha^{(1)},\ldots,\alpha^{(N)}]$ consists
of columns $\alpha^{(i)}\in\R^{K^{(i)}}$.

In fact, $\mathcal{F}_{0}[\boldsymbol{\alpha}]$ can be viewed as
a minorization of $\mathcal{F}[\boldsymbol{\Lambda}]$ under the correspondence
\begin{equation}
\Lambda^{(i)}=\Lambda_{0}^{(i)}+\underbrace{\sum_{k=1}^{K^{(i)}}\alpha_{k}^{(i)}u_{k}^{(i)}u_{k}^{(i)\top}}_{\Delta\Lambda^{(i)}}.\label{eq:alphaLambda}
\end{equation}
 To see this point, observe that by the Golden-Thompson inequality,
\begin{align*}
\Tr\left[e^{-\beta\left(C-\bigoplus_{i=1}^{N}\Lambda^{(i)}\right)}\right] & \leq\Tr\left[e^{\beta\bigoplus_{i=1}^{N}\Delta\Lambda^{(i)}}e^{-\beta\left(C-\bigoplus_{i=1}^{N}\Lambda_{0}^{(i)}\right)}\right]\\
 & =\sum_{i=1}^{N}\Tr\left[e^{\beta[\Delta\Lambda^{(i)}]}X_{0}^{(i,i)}\right].
\end{align*}
Substituting $\Delta\Lambda^{(i)}=\sum_{k=1}^{K^{(i)}}\alpha_{k}^{(i)}u_{k}^{(i)}u_{k}^{(i)\top}$,we
obtain our proxy objective (\ref{eq:Fproxy}).

Now consider the optimizing the variable $\alpha_{k}^{(i)}$, holding
all other $\alpha_{l}^{(j)}=0$ fixed at zero. It is equivalent to
maximize 
\[
\alpha_{k}^{(i)}-\beta^{-1}\Tr\left[e^{\beta\alpha_{k}^{(i)}u_{k}^{(i)}u_{k}^{(i)\top}}X_{0}^{(i,i)}\right].
\]
 Compute the rank-one exponential $e^{\beta\alpha_{k}^{(i)}u_{k}^{(i)}u_{k}^{(i)\top}}=I_{K^{(i)}}+(e^{\beta\alpha_{k}^{(i)}}-1)u_{k}^{(i)}u_{k}^{(i)\top}$,
from which it follows that it is equivalent in turn to maximize the
(evidently concave) objective 
\[
\alpha_{k}^{(i)}-\beta^{-1}e^{\beta\alpha_{k}^{(i)}}u_{k}^{(i)\top}X_{0}^{(i,i)}u_{k}^{(i)}
\]
 with respect to $\alpha_{k}^{(i)}$. The first-order optimality condition
then yields 
\[
\alpha_{k}^{(i)}=-\beta^{-1}\log\left(u_{k}^{(i)\top}X_{0}^{(i,i)}u_{k}^{(i)}\right).
\]

Hence if we compute the coordinate ascent update for each $\alpha_{k}^{(i)}$
independently (holding all other variables fixed at zero) and then
accept all these updates simultaneously (i.e., `Jacobi style'), we
precisely recover the dual update (\ref{eq:dualupdate}) formula specified
above.

By contrast, a plain coordinate ascent approach derived by substituting,
e.g., a random unit vector $u^{(i)}$ in the role $u_{k}^{(i)}$ in
the above, could guarantee that the dual objective is increasing (while
still performing simultaneous updates in the direction of $u^{(i)}u^{(i)\top}$
over all $i=1,\ldots,N$). However, the Jacobi-style approach is more
convenient since the full-rank matrix update $\Delta\Lambda^{(i)}$
can be computed more straightforwardly with the same batched matrix-vector
multiplications by the effective cost matrix $C-\bigoplus_{i=1}^{N}\Lambda_{0}^{(i)}$.
Moreover, we find that the Jacobi-style approach is efficient in practice.

\subsubsection{Trace estimation \label{sec:trace}}

In order to implement the dual update (\ref{eq:dualupdate2}) in practice,
we must estimate the diagonal blocks $X_{\beta,\boldsymbol{\Lambda}}^{(i,i)}$
of the matrix $X_{\beta,\boldsymbol{\Lambda}}$ defined by (\ref{eq:XbetaLambda}).
Unfortunately, direct evaluation of the entire matrix exponential
in (\ref{eq:XbetaLambda}) is far too costly, costing $\Omega(K^{3}N^{3})$
operations in general, where $K:=\max_{i=1,\ldots,N}K^{(i)}$.

However, using randomized trace estimation, relying only on matrix-vector
multiplications (matvecs) by the effective cost matrix 
\[
C_{\mathrm{eff}}[\Blam]:=C-\bigoplus_{i=1}^{N}\Lambda^{(i)},
\]
 we can construct estimates of these diagonal blocks much more efficiently
than the entire matrix itself. Note that by exploiting the sparsity
pattern of $C_{\mathrm{eff}}$, we can perform each such matvec in
$O(KN^{2}+NK^{2})$ operations. Ignoring log factors, it will require
$\tilde{O}(K)$ matvecs to recover the full block diagonal to fixed
relative spectral accuracy with high probability, and therefore the
full construction of the block diagonal will cost $\tilde{O}(K^{2}N^{2}+NK^{3})$
operations, where $\tilde{O}$ indicates the omission of log factors.

To wit, let $Z\in\R^{L\times S}$ be a matrix with i.i.d. standard
Gaussian entries. Then we can `approximate' $X_{\beta,\Blam}$ with
the estimator $\hat{X}_{\beta,\Blam}$ obtained by inserting a stochastic
resolution of identity $\frac{1}{S}ZZ^{\top}$ as follows:
\[
\hat{X}_{\beta,\Blam}=X_{\beta,\Blam}^{1/2}ZZ^{\top}X_{\beta,\Blam}^{1/2}.
\]

We never form $\hat{X}_{\beta,\Blam}$ directly but instead form its
diagonal blocks $\hat{X}_{\beta,\Blam}^{(i,i)}$ as follows. First
form: 

\[
W=X_{\beta,\Blam}^{1/2}Z\in\R^{L\times S},
\]
 requiring $S$ matvecs by 
\[
X_{\beta,\Blam}^{1/2}=e^{-\frac{\beta}{2}C_{\mathrm{eff}}[\Blam]}.
\]
 The implementation of these matvecs can be reduced to matvecs by
the sparse matrix $C_{\mathrm{eff}}$ itself, using either the algorithm
of~\cite{expmv} or Chebyshev expansion~\cite{trefethen2019approximation}. For
fixed $\beta$ , the number of matvecs by $C_{\mathrm{eff}}$ required
per matvec by $X_{\beta,\Blam}^{1/2}$ is a constant, independent
of problem size. The dependence on $\beta$ is difficult to understand
in the algorithm~\cite{expmv}, though a naive approach
based on Trotter splitting~\cite{Trotter1959}, suggests that $O(\beta)$ is
easily attainable. Meanwhile, in Chebyshev expansion, the number grows
as $O(\sqrt{\beta})$ as can be verified by standard Chebyshev approximation
bounds~\cite{trefethen2019approximation}.

Once $W$ is formed, we can construct our block-diagonal estimates
as 
\[
\hat{X}_{\beta,\Blam}^{(i,i)}=\frac{1}{S}W^{(i)}W^{(i)\top},
\]
 where $W^{(i)}\in\R^{K^{(i)}\times S}$ is the $i$-th block of $W$.

The total complexity of the construction is therefore $O(KN^{2}S+NK^{2}S)$
for fixed $\beta$. The following concentration bound guarantees that
taking $S=O(K\log K)$ is sufficient for a good relative spectral
approximation of $X_{\beta,\Blam}^{(i,i)}$ with high probability,
which is what we need to guarantee a suitable error bound on our estimate
$\log\hat{X}_{\beta,\Blam}^{(i,i)}$ of $\log X_{\beta,\Blam}^{(i,i)}$,
justifying its use within the optimizer.
\begin{prop}
\label{prop:concentration}Let $\ve,\delta\in(0,1)$. With notation
as in the preceding discussion, there exists a constant $C>0$ such
that if $S\geq\frac{C\,\log(N/\delta)}{\ve^{2}}\,K\log K$, then 

\[
(1-\ve)\,X_{\beta,\Blam}^{(i,i)}\preceq\hat{X}_{\beta,\Blam}^{(i,i)}\preceq(1+\ve)\,X_{\beta,\Blam}^{(i,i)}
\]
 with probability at least $\delta$. It follows that 
\[
\Vert\log\hat{X}_{\beta,\Blam}^{(i,i)}-\log X_{\beta,\Blam}^{(i,i)}\Vert\leq\vert\log(1-\ve)\vert,
\]
 where $\Vert\,\cdot\,\Vert$ denotes the spectral norm.
\end{prop}

\begin{rem}
If we assume that $\ve\le\frac{1}{2}$, for example, then $\vert\log(1-\ve)\vert\leq1.4\times\ve=O(\ve)$.
Therefore, with $S=\tilde{O}(K/\ve^{2})$ random vectors, we attain
$O(\ve)$ spectral error for the estimated block logarithms.
\end{rem}

\begin{proof}
For any fixed $i$, note that $\hat{X}_{\beta,\Blam}^{(i,i)}=(Z^{\top}A)^{\top}(Z^{\top}A)$,
where $A:=\left[X_{\beta,\Blam}^{1/2}\right]^{(:,i)}$. Then $\tilde{A}:=Z^{\top}A$
is a Johnson-Lindenstrauss random projection of $A$, and standard
results~\cite{Martinsson_Tropp_2020} guarantee that this random projection is a subspace
embedding, i.e., 
\[
(1-\ve)\,A^{\top}A\preceq\tilde{A}^{\top}\tilde{A}\preceq(1+\ve)\,A^{\top}A
\]
 with probability at least $\delta$, provided that $S\geq\frac{C\,\log(1/\delta)}{\ve^{2}}\,K\log K$
for a suitable constant $C$. Since we want this property to hold
for, we substitute $\delta\leftarrow\frac{\delta}{N}$ and apply the
union bound over $i=1,\ldots,N$ to obtain the desired result.

Then by the matrix-monotonicity of the logarithm~\cite{Horn_Johnson_1991}, it follows
that 
\[
\log(1-\ve)\,\mathbf{I}_{K^{(i)}}+\log X_{\beta,\Blam}^{(i,i)}\preceq\log\hat{X}_{\beta,\Blam}^{(i,i)}\preceq\log(1+\ve)\,\mathbf{I}_{K^{(i)}}+\log X_{\beta,\Blam}^{(i,i)}.
\]
 Since $\vert\log(1+\ve)\vert\leq\vert\log(1-\ve)\vert$, the second
statement of the proposition follows.
\end{proof}

\subsubsection{Algorithm summary}

The algorithm outlined above for solving the strong regularized SDP
(\ref{eq:strongSDPreg}) is summarized in Algorithm \ref{alg:strong}.
Observe that we include a damping parameter $\gamma>0$ which controls
a schedule for tapering off the step size, and in practice we find
that $\gamma=5$ works quite well. Note that $\gamma=+\infty$ recovers
the optimizer with no damping.

\begin{algorithm}
\caption{Optimizer for strong regularized SDP (\ref{eq:strongSDPreg}) }

\begin{algorithmic}[1]
\Require{Cost matrix $C=-Q$, regularization parameter $\beta > 0$, randomization parameter $S$, damping parameter $\gamma > 0$}
\State{Initialize $\Lambda^{(i)}=0_{K^{(i)}\times K^{(i)}}$ for $i=1,\ldots,N$, t=0}
\While{$\Blam=(\Lambda^{(1)},\ldots,\Lambda^{(N)})$ not converged}
\State{Set $t \leftarrow t+1$ and $\eta \leftarrow \min(\gamma / t,1)$}
\State{Let $Z\in\R^{L\times S}$ have i.i.d. standard Gaussian entries}
\State{Form $C_{\mathrm{eff}}[\Blam]:=C-\bigoplus_{i=1}^{N}\Lambda^{(i)}$ as an $L\times L$ sparse matrix}
\State{Form $W=e^{-\frac{\beta}{2}C_{\mathrm{eff}}[\Blam]}Z\in\R^{L\times S}$ using the algorithm of~\cite{expmv} or Chebyshev expansion}
\State{Form $B^{(i)}=\frac{1}{S}W^{(i)}W^{(i)\top}\in\R^{K^{(i)}\times K^{(i)}}$ for $i=1,\ldots,N$}
\State{Set $\Lambda^{(i)}\leftarrow\Lambda^{(i)}-\eta \beta^{-1}\log\left(B^{(i)}\right)$ for $i=1,\ldots,N$}
\EndWhile
\end{algorithmic}\label{alg:strong}
\end{algorithm}

By Proposition \ref{prop:concentration} we can guarantee $\ve$-error
in the spectral norm of our updates for each $\Lambda^{(i)}$ with
$S=\tilde{O}(K/\ve^{2})$. In this sense the total cost (for fixed
$\beta$) of an $\ve$-accurate iteration is $\tilde{O}((N^{2}K^{2}+NK^{3})/\ve^{2})$.
Elsewhere in the manuscript we think of $\ve$ as constant in our
scaling considerations.

\subsection{Weak regularized SDP}

The dual problem to (\ref{eq:weakSDPreg}) can be derived as 
\[
\underset{\lambda\in\R^{L},\,\mu\in\R^{N}}{\text{maximize}}\ \ \mathcal{F}[\lambda,\mu],
\]
 where 
\[
\mathcal{F}[\lambda,\mu]=\mathbf{1}_{L}\cdot\lambda+\mathbf{1}_{N}\cdot\mu-\beta^{-1}\Tr\left[e^{-\beta\left(C-\mathrm{diag}(\lambda)-\bigoplus_{i=1}^{N}\mu^{(i)}\Sigma^{(i)}\right)}\right],
\]
 where we have defined 
\[
\Sigma^{(i)}:=\frac{\mathbf{1}_{K^{(i)}}\mathbf{1}_{K^{(i)}}^{\top}}{K^{(i)}}.
\]

The primal solution can be recovered as 
\[
X_{\beta,\lambda,\mu}:=e^{-\beta\left(C-\mathrm{diag}(\lambda)-\bigoplus_{i=1}^{N}\mu^{(i)}\Sigma^{(i)}\right)}.
\]

\subsubsection{Optimizer }

Our optimizer is defined by the update: 
\begin{align}
 & X_{0}\leftarrow X_{\beta,\lambda,\mu}\label{eq:dualupdateweak}\\
 & \lambda\leftarrow\lambda-\beta^{-1}\log\left(\mathrm{diag}[X_{0}]\right)\nonumber \\
 & \mu^{(i)}\leftarrow\mu^{(i)}-\beta^{-1}\log\left(\frac{\mathbf{1}_{K^{(i)}}^{\top}X_{0}^{(i,i)}\mathbf{1}_{K^{(i)}}}{K^{(i)}}\right),\quad i=1,\ldots,N.\nonumber 
\end{align}
 Later on (cf. Algorithm \ref{alg:weak} below), we will present the
algorithm more formally and include a slight extension allowing for
damping of the step size.

Given current guess $\lambda_{0},\mu_{0}$, the first step can be
viewed as an exact update of $\lambda$ for the minorized objective
(cf. \cite{lindsey2023fastrandomizedentropicallyregularized}): 
\[
\mathcal{F}_{0}[\lambda]=\mathbf{1}_{L}\cdot\lambda+\mathbf{1}_{N}\cdot\mu_{0}-\beta^{-1}\Tr\left[e^{\beta(\lambda-\lambda_{0})}e^{-\beta\left(C-\mathrm{diag}(\lambda_{0})-\bigoplus_{i=1}^{N}\mu_{0}^{(i)}\Sigma^{(i)}\right)}\right].
\]
 We have left the right-hand side unsimplified to illustrate the application
of the Golden-Thompson inequality.

Meanwhile, the second step can be viewed as an exact update of $\mu$
for the the minorized objective (cf. the derivation in Section \ref{sec:minorization})
: 
\[
\mathcal{G}_{0}[\mu]=\mathbf{1}_{L}\cdot\lambda_{0}+\mathbf{1}_{N}\cdot\mu-\beta^{-1}\Tr\left[e^{\beta\bigoplus_{i=1}^{N}\left(\mu^{(i)}-\mu_{0}^{(i)}\right)\Sigma^{(i)}}e^{-\beta\left(C-\mathrm{diag}(\lambda_{0})-\bigoplus_{i=1}^{N}\mu_{0}^{(i)}\Sigma^{(i)}\right)}\right].
\]

Thus either the second or third line individually of (\ref{eq:dualupdateweak})
must increase the dual objective. Instead of alternating these updates,
we compute them `Jacobi-style' in terms of the same matrix $X_{0}$,
in order to economize on matvecs in the randomized trace estimation
required for these updates.

\subsubsection{Trace estimation}

In order to implement (\ref{eq:dualupdateweak}) we must devise estimators
for $\mathrm{diag}[X_{\beta,\lambda,\mu}]$, as well as $\frac{\mathbf{1}_{K^{(i)}}^{\top}X_{\beta,\lambda,\mu}^{(i,i)}\mathbf{1}_{K^{(i)}}}{K^{(i)}}$,
for $i=1,\ldots,N$.

Again we define
\[
\hat{X}_{\beta,\lambda,\mu}:=\frac{1}{S}\left[X_{\beta,\lambda,\mu}^{1/2}ZZ^{\top}X_{\beta,\lambda,\mu}^{1/2}\right],
\]
 where $Z\in\R^{L\times S}$ has i.i.d. Gaussian entries.

Then we simply define our estimators as $\mathrm{diag}[\hat{X}_{\beta,\lambda,\mu}]$
and $\frac{\mathbf{1}_{K^{(i)}}^{\top}\hat{X}_{\beta,\lambda,\mu}^{(i,i)}\mathbf{1}_{K^{(i)}}}{K^{(i)}}$,
respectively. Note that in practice we never form $\hat{X}_{\beta,\lambda,\mu}$
explicitly, but instead form 
\[
W=X_{\beta,\lambda,\mu}^{1/2}Z
\]
 as above using $S$ matvecs by $X_{\beta,\lambda,\mu}^{1/2}$ and
then directly compute 
\[
\mathrm{diag}[\hat{X}_{\beta,\lambda,\mu}]=\left[W\odot W\right]\frac{\mathbf{1}_{S}}{S}
\]
 and 
\[
\frac{\mathbf{1}_{K^{(i)}}^{\top}\hat{X}_{\beta,\lambda,\mu}^{(i,i)}\mathbf{1}_{K^{(i)}}}{K^{(i)}}=\frac{1}{K^{(i)}}\left[\left(W^{(i)\top}\mathbf{1}_{K^{(i)}}\right)\odot\left(W^{(i)\top}\mathbf{1}_{K^{(i)}}\right)\right]\cdot\frac{\mathbf{1}_{S}}{S}.
\]

Defining 
\[
C_{\mathrm{eff}}[\lambda,\mu]:=C-\mathrm{diag}(\lambda)-\bigoplus_{i=1}^{N}\mu^{(i)}\Sigma^{(i)},
\]
 we can reduce the task of performing one matvec by $X_{\beta,\lambda,\mu}^{1/2}=e^{-\frac{\beta}{2}C_{\mathrm{eff}}[\lambda,\mu]}$
(just as in Section \ref{sec:trace}) to a constant number of matvecs
by $C_{\mathrm{eff}}=C_{\mathrm{eff}}[\lambda,\mu]$. Observe that
a matvec by $C_{\mathrm{eff}}$ can be performed in $O(N^{2}K)$ operations.
Indeed, note that a matvec by the last term in $C_{\mathrm{eff}}$
can be performed as 
\[
\left[\bigoplus_{i=1}^{N}\mu^{(i)}\Sigma^{(i)}\right]z=\bigoplus_{i=1}^{N}\frac{\mu^{(i)}}{K^{(i)}}\mathbf{1}_{K^{(i)}}\mathbf{1}_{K^{(i)}}^{\top}z^{(i)}
\]
 with $O(NK)$ cost. In particular, avoiding the treatment of $\mathbf{1}_{K^{(i)}}\mathbf{1}_{K^{(i)}}^{\top}$
as a dense matrix allows us to avoid any $O(NK^{2})$ contribution
to the scaling, by contrast to Section \ref{sec:trace}. 

Meanwhile, the number $S$ of matvecs that we require \emph{also}
scales better than the number required in Section \ref{sec:trace}
(as analyzed in Proposition \ref{prop:concentration}). Indeed, the
following proposition guarantees that $S=\tilde{O}(1)$ suffices,
meaning that one iteration of the optimizer can be performed with
the nearly optimal scaling of $\tilde{O}(\mathrm{nnz}(C))=O(N^{2}K)$.
\begin{prop}
\label{prop:concentration2} Let $\ve,\delta\in(0,1)$. With notation
as in the preceding discussion, there exists a constant $C>0$ such
that if $S\geq\frac{C\,\log(L/\delta)}{\ve^{2}}$, then with probability
at least $\delta$, we have that both 

\[
(1-\ve)\,\mathrm{diag}[X_{\beta,\lambda,\mu}]\leq\mathrm{diag}[\hat{X}_{\beta,\lambda,\mu}]\leq(1+\ve)\,\mathrm{diag}[X_{\beta,\lambda,\mu}]
\]
 and 
\[
(1-\ve)\,\mathbf{1}_{K^{(i)}}^{\top}X_{\beta,\lambda,\mu}^{(i,i)}\mathbf{1}_{K^{(i)}}\leq\mathbf{1}_{K^{(i)}}^{\top}\hat{X}_{\beta,\lambda,\mu}^{(i,i)}\mathbf{1}_{K^{(i)}}\leq(1+\ve)\,\mathbf{1}_{K^{(i)}}^{\top}X_{\beta,\lambda,\mu}^{(i,i)}\mathbf{1}_{K^{(i)}}
\]
 for all $i=1,\ldots,N$.
\end{prop}

\begin{rem}
As was the case in Proposition \ref{prop:concentration}, these relative
error bounds imply absolute error bounds for the logarithms used within
our optimizers. (In fact, this point is much simpler here because
the logarithms are scalar, i.e., we do not need to use matrix-monotonicity.)
\end{rem}

\begin{proof}
The first statement follows from the argument of Proposition \ref{prop:concentration}
in the case where $K^{(i)}=1$ for all $i$.

For the second statement, for any fixed $i$ we observe that 
\[
\mathbf{1}_{K^{(i)}}^{\top}\hat{X}_{\beta,\lambda,\mu}^{(i,i)}\mathbf{1}_{K^{(i)}}=(Z^{\top}a)^{\top}(Z^{\top}a),
\]
 where 
\[
a:=\left[X_{\beta,\Blam}^{1/2}\right]^{(:,i)}\mathbf{1}_{K^{(i)}}\in\R^{L\times1}.
\]
 Then $\tilde{a}:=Z^{\top}a$ is a Johnson-Lindenstrauss random projection
of $a$, and standard results~\cite{Martinsson_Tropp_2020} guarantee the subspace embedding
property 
\[
(1-\ve)\,a^{\top}a\leq\tilde{a}^{\top}\tilde{a}\leq(1+\ve)\,a^{\top}a.
\]
for $S=\Omega\left(\frac{\log(1/\delta)}{\ve^{2}}\right)$. Then the
union bound over $i=1,\ldots,N\leq L$ guarantees that the second
statement to be shown.
\end{proof}

\subsubsection{Algorithm summary}

The algorithm outlined above for solving the weak regularized SDP
(\ref{eq:weakSDPreg}) is summarized in Algorithm \ref{alg:weak}.
Note once again the inclusion of the damping parameter $\gamma>0$.

\begin{algorithm}
\caption{Optimizer for weak regularized SDP (\ref{eq:weakSDPreg})}

\begin{algorithmic}[1]
\Require{Cost matrix $C=-Q$, regularization parameter $\beta > 0$, randomization parameter $S$, damping parameter $\gamma > 0$}
\State{Initialize $\lambda=0_{L}$, $\mu=0_{N}$ }
\While{$\lambda$, $\mu$ not converged}
\State{Set $t \leftarrow t+1$ and $\eta \leftarrow \min(\gamma / t,1)$}
\State{Let $Z\in\R^{L\times S}$ have i.i.d. standard Gaussian entries}
\State{Construct a matvec oracle for  \[ C_{\mathrm{eff}}[\lambda,\mu]:=C-\mathrm{diag}(\lambda)-\bigoplus_{i=1}^{N}\mu^{(i)} \Sigma^{(i)} \]}
\State{Form $W=e^{-\frac{\beta}{2}C_{\mathrm{eff}}[\lambda, \mu]}Z\in\R^{L\times S}$ using the algorithm of \cite{expmv} or Chebyshev expansion}
\State{Form $w^{(i)}=W^{(i)\top}\mathbf{1}_{K^{(i)}}/\sqrt{K^{(i)}}\in\R^{S}$ for each $i=1,\ldots,N$}
\State{Form $b=\left[W\odot W\right]\frac{\mathbf{1}_{S}}{S}$ and $b^{(i)}=$ $\left[w^{(i)}\odot w^{(i)}\right]\cdot\frac{\mathbf{1}_{S}}{S}$ for each $i=1,\ldots,N$}
\State{Set $\lambda\leftarrow\lambda- \eta \beta^{-1}\log\left(b\right)$ and $\mu^{(i)}\leftarrow\mu^{(i)}- \eta \beta^{-1}\log(b^{(i)})$ for $i=1,\ldots,N$}
\EndWhile
\end{algorithmic}\label{alg:weak}
\end{algorithm}

By Proposition \ref{prop:concentration2} we can guarantee $\ve$-error
(uniformly entrywise) in our updates of $\mu,\lambda$ with $S=\tilde{O}(1/\ve^{2})$.
In this sense the total cost (for fixed $\beta$) of an $\ve$-accurate
iteration is $\tilde{O}(\text{nnz}(C)/\ve^{2})$. Elsewhere in the
manuscript we think of $\ve$ as constant in our scaling considerations.

\section{Recovery \label{sec:recovery}}

Next we discuss how to recover a collection of correspondences from
keypoints to registry points, i.e., row-partial permutation matrices
$P^{(i)}\in\R^{K^{(i)}\times M}$ for $i=1,\ldots,M$, from our dual
solution. Automatically we then obtain cycle-consistent partial matchings between keypoints in all pairs of images.

Note carefully that we do not want this recovery to impose
a computational bottleneck relative to the optimizer itself. Therefore
the computational considerations are different for the case of the
strong regularized SDP (\ref{eq:strongSDPreg}) and the weak regularized
SDP (\ref{eq:weakSDPreg}), motivating the `slow' and `fast' recovery algorithms that we present below. For each of these scenarios,
the case of full permutation synchronization is simpler, so we present it first.

Finally, we introduce an alternative `masked recovery' procdure which does not output a cycle-consistent correspondences, but rather attempts to individually classify observed correspondences as correct or incorrect.

Throughout, we adopt the notation $[n]:=\{1,\ldots,n\}$ for arbitrary
positive integer $n$.

\subsection{Slow recovery}

Let $X=X_{\beta,\Blam}$ denote the primal matrix associated with
the dual optimizer $\Blam$ of the strong regularized SDP (\ref{eq:strongSDPreg})
. We never want to form $X$ explicitly. Instead we want to access
it only via a matvec oracle which can be implemented as above using
the algorithm of~\cite{expmv} or Chebyshev expansion. To avoid imposing
a bottleneck, we merely need to achieve recovery with $\tilde{O}(N^{2}K^{2}+NK^{3})$
cost. In particular we can call the matvec oracle $\tilde{O}(K)$
times.

We will see that this is straightforwardly possible in the case of
full permutation recovery. Meanwhile, in the case of partial permutation
recovery, with the same cost we can recover, for a single image, all
correspondences between keypoints in this image and those of all other
images.

To fully register all keypoints in all images, the recovery cost depends
somewhat subtly on the amount of overlap between distinct images and,
relatedly, the total number $M$ of ground truth registry points.
Under the natural assumption that each of the registry points appears
in a fixed fraction of at least $\tilde{\Omega}(K/M)$ of the images,
the approach described in this section can achieve recovery with $\tilde{O}(N^{2}KM+NK^{2}M)$
additional cost. Therefore if $M/K=O(1)$, the scaling of this recovery
approach is no worse than that of the optimization approach for (\ref{eq:strongSDPreg}).

\subsubsection{Full permutation recovery}

The recovery in this case is simplest. Recall that in this setting
we have $K^{(i)}=K$ for all $i=1,\ldots,N$, and we seek to recover
$X\approx PP^{\top}$, where the blocks $P^{(i)}\in\R^{K\times K}$
of $P$ are \emph{full} permutation matrices.

To achieve our aim, define 
\begin{equation}
E^{(j)}=\left(\begin{array}{c}
0_{K\times K}\\
\vdots\\
0_{K\times K}\\
\mathbf{I}_{K}\\
0_{K\times K}\\
\vdots\\
0_{K\times K}
\end{array}\right)\leftarrow\text{\ensuremath{j}-th block}\label{eq:Ej0}
\end{equation}
 for any $j\in[N]$. In particular the shape of $E^{(j)}$ is $NK\times K$.

Observe that the $j$-th block-column of $X$, 
\[
X^{(:,j)}=XE^{(j)},
\]
 can be formed using $K$ matvecs by $X$. If we assume $X\approx PP^{\top}$,
then 
\[
X^{(:,j)}\approx\left(\begin{array}{c}
P^{(1)}P^{(j)\top}\\
\vdots\\
P^{(N)}P^{(j)\top}
\end{array}\right).
\]

Taking $j=1$ for simplicity and assuming without loss of generality
that $P^{(1)}=\mathrm{\mathbf{I}_{K}}$, we can directly recover 
\[
X^{(:,1)}\approx\left(\begin{array}{c}
\mathbf{I}_{K}\\
P^{(2)}\\
\vdots\\
P^{(N)}
\end{array}\right).
\]
Then we can we can use the Hungarian algorithm~\cite{Kuhn_1955} to round
each block $X^{(i,1)}$ of $X^{(:,1)}$ to the nearest permutation
matrix with respect to the Frobenius norm, which, over all $i=1,\ldots,N$,
costs $O(NK^{3})$ in total.

\subsubsection{Partial permutation recovery \label{sec:partialrecov}}

Now possibly the $K^{(i)}$ are distinct, and we seek to recover $X\approx PP^{\top}$,
where the blocks $P^{(i)}\in\R^{K^{(i)}\times M}$ of $P$ are row-partial
permutation matrices. Moreover, the number $M$ of registry points
is \emph{a priori }unknown.

In this setting it is useful to define, more generally than in the
previous section, 

\begin{equation}
E^{(j)}=\left(\begin{array}{c}
0_{K^{(1)}\times K^{(j)}}\\
\vdots\\
0_{K^{(j-1)}\times K^{(j)}}\\
\mathbf{I}_{K^{(j)}}\\
0_{K^{(j+1)}\times K^{(j)}}\\
\vdots\\
0_{K^{(N)}\times K^{(j)}}
\end{array}\right)\in\R^{L\times K^{(j)}}\label{eq:Ej}
\end{equation}
 for arbitrary $j\in[N]$.

Then 
\[
X^{(:,j)}=XE^{(j)}\approx\left(\begin{array}{c}
P^{(1)}P^{(j)\top}\\
\vdots\\
P^{(N)}P^{(j)\top}
\end{array}\right).
\]
 can again be formed using $O(K)$ matvecs by $X$.

We cannot hope to recover $P$ all at once, but for any fixed $i\neq j$,
we can recover all keypoint correspondences between images $i$ and
$j$. It is useful to adopt the indexing convention $(i,k)$ for the
$k$-th keypoint in image $i$. Let 
\[
\mathcal{K}:=\left\{ (i,k)\,:\,i\in[N],\,k\in[K^{(i)}]\right\} 
\]
 denote the set of all keypoints.

Now fix $i\neq j$, and let $e_{l}^{(j)}\in\R^{1\times K^{(j)}}$
($l=1,\ldots,K^{(j)}$) denote the standard basis vectors in $\R^{K^{(j)}}$,
treated as row vectors, with the convention that $e_{0}^{(j)}=0_{1\times K^{(j)}}$.
Then after extracting the block $X^{(i,j)}$, we define (only tentatively,
since we will revise our approach shortly): 
\[
l_{ij}(k)=\underset{l\in\{0\}\cup[K^{(j)}]}{\mathrm{argmin}}\Vert X_{k,:}^{(i,j)}-e_{l}^{(j)}\Vert,
\]
 for each $k=1,\ldots,K^{(i)}$. If $l_{ij}(k)=0$, then no identification
is made yet for keypoint $(i,k)$. However, if $l_{ij}(k)\neq0$,
we identify a correspondence $(i,k)\sim(j,l_{ij}(k))$. These correspondences
will generate an equivalence relation.

Proceeding more carefully, we should force the correspondences $k\mapsto l_{ij}(k)$
to be injective (ignoring the unpaired $k$ for which $l_{ij}(k)=0$)
by determining $l_{ij}(k)$ sequentially for $k=1,\ldots,K^{(i)}$,
restricting each successive $\mathrm{argmin}$ to the subset of target
indices not yet claimed (together with the zero index).

Now we have seen that for any fixed $j$, using $O(K)$ matvecs, we
may recover all possible identifications of the form $(j,l)\sim(i,k)$.
However, to recover all identifications over all $j$, without the
possibility of introducing (by transitivity of our equivalence relation)
any identifications of keypoints within the same image, we must adopt
the more careful procedure outlined in Algorithm \ref{alg:strongpartialrecovery}. 

\begin{algorithm}
\caption{Slow partial permutation recovery}

\begin{algorithmic}[1]

\Require{Matvec oracle for the primal solution $X$}

\State{Initialize $m=0$ and $\mathcal{S}^{(i)}=[K^{(i)}]$ for $i=1,\ldots,N$}

\While{there exists $i\in[N]$ such that $\mathcal{S}^{(i)}\neq\emptyset$}

\State{Set $j\leftarrow\underset{j'\in[N]}{\mathrm{argmin}}\left\{ \sum_{i=1}^{N}\Vert Q_{\mathcal{S}^{(i)},\mathcal{S}^{(j')}}^{(i,j')}\Vert_{\mathrm{F}}^{2}\right\}$, or alternatively choose uniformly randomly from indices not previously chosen \label{state:choice}}

\State{Form $E^{(j)}\in\R^{L\times K^{(j)}}$ following (\ref{eq:Ej})}

\State{Form $X^{(:,j)}=XE^{(j)}$ using $K^{(j)}$ matvecs by $X$}

\For{$l\in\mathcal{S}^{(j)}$}

\State{Set $m\leftarrow m+1$}
\State{Set $\mathcal{R}(j,l)\leftarrow m$}

\EndFor

\For{$i\in[N]\backslash\{j\}$}

\State{Set $\mathcal{T} \leftarrow \mathcal{S}^{(j)}$}

\For{$k\in\mathcal{S}^{(i)}$}

\State{Set $l\leftarrow\underset{l\in\{0\}\cup \mathcal{T} }{\mathrm{argmin}}\left\{ \Vert X_{k,:}^{(i,j)}-e_{l}^{(j)}\Vert\right\}$}

\If{$l\neq 0$}
\State{Set $\mathcal{S}^{(i)}\leftarrow\mathcal{S}^{(i)}\backslash\{k\}$}
\State{Set $\mathcal{T} \leftarrow \mathcal{T} \backslash \{l\}$}
\State{Set $\mathcal{R}(i,k)\leftarrow\mathcal{R}(j,l)$}
\EndIf

\EndFor
\EndFor

\State{$\mathcal{S}^{(j)}\leftarrow\emptyset$}

\EndWhile

\State{Set $M \leftarrow m$}

\State{\Return{ $\mathcal{R}:\mathcal{K}\ra [M]$ }}
\end{algorithmic}\label{alg:strongpartialrecovery}
\end{algorithm}

Intuitively, the algorithm functions as follows. We maintain for each
image $i$ a list of keypoints $\mathcal{S}^{(i)}$ that have not
yet been `registered,' or identified with a registry point. We remove
keypoints from these lists as registrations are made, and once removed
they become ineligible for any future registration. 

At every stage, we pick the image $j$ which maximizes the total number
of correspondences between its unregistered keypoints and unregistered
keypoints across all remaining images, measured according to our input
data $Q$. (Alternatively, we can just choose $j$ uniformly randomly
from all images not previously chosen.) We will find and register
all keypoints in all images $i\in[N]\backslash j$ that correspond
to some keypoint in $j$. We will assign identified images to the
same registry point.

To accomplish this, first we must expand our registry set to accommodate
points seen for the first time in the $j$-th image. We loop over
all unregistered keypoints in the $j$-th image and assign them to
new registry points, incrementing our counter $m$ for the number
of distinct registry points discovered so far.

Then we form the $j$-th block column $X^{(:,j)}$ as above and sequentially
determine all possible identifications of the form $(i,k)\sim(j,l)$
between previously unregistered keypoints, successively removing any
identified keypoints $(i,k)$ from $\mathcal{S}^{(i)}$ and maintaining
`injectivity,' i.e., that we do not have $(i,k)\sim(j,l)$ and $(i,k')\sim(j,l)$
for $k\neq k'$. When we remove a keypoint $(i,k)$, we assign it
to the same registry point as its identified partner $(j,l$) in the
$j$-th image.

Then we also set $\mathcal{S}^{(j)}\leftarrow\emptyset$, indicating
that all keypoints in the $j$-th image have now been registered.

We continue in this fashion, selecting a new image index $j$ and
registering all keypoints corresponding to any keypoint in the $j$-th
image, until all keypoints have been eliminated. (Note that, since
$\mathcal{Q}^{(i,i)}=\mathbf{I}_{K^{(i)}}$ for all $i$, the objective
used in Step \ref{state:choice} is zero if and only if $\mathcal{S}^{(j')}=\emptyset$.
Therefore we never repeat any index $j$ used in an earlier stage.)

Ultimately the algorithm returns a number $M$ of registry points
and registration map 
\[
\mathcal{R}:\mathcal{K}\ra[M]
\]
 which assigns each keypoint $(i,k)\in\mathcal{K}$ to a registry
point, such that moreover $\mathcal{R}(i,\,\cdot\,)$ is injective
for every $i$. Such a map can be identified with a collection $P^{(i)}\in\R^{K^{(i)}\times M}$,
$i=1,\ldots,N$, of row-partial permutation matrices via $P_{kl}^{(i)}=\delta_{k,\mathcal{R}(i,k)}$.

To assess the scaling of the recovery algorithm, it suffices to bound
the number $J$ of images that must be processed before the algorithm
terminates. It is easier to analyze the random selection approach.
Let us assume that each registry point appears in a fraction of at
least $f$ of the images. Then if images are chosen randomly, the
probability that there exists a registry point not contained in any
of the first $J$ images is bounded (via the union bound) by $M(1-f)^{J}\leq Me^{-Jf}$.
It follows that if 
\[
J\geq\frac{\log(M/\delta)}{f},
\]
 then, with probability at least $1-\delta$, all registry points
have been seen in at least one of the first $J$ images, meaning that
the algorithm terminates in $J$ steps.

Thus the expected total cost of the recovery is $\tilde{O}(N^{2}K^{2}/f+NK^{3}/f)$.
Under the natural assumption that $f=\tilde{O}(K/M)$, we obtain a
recovery cost of $\tilde{O}(N^{2}KM+NK^{2}M)$.

\subsection{Fast recovery}

Let $X=X_{\beta,\lambda,\mu}$ denote the primal matrix associated
with the dual optimizer $\lambda,\mu$ of the weak regularized SDP
(\ref{eq:weakSDPreg}). As above, we never want to form $X$ explicitly
and only access it via a matevec oracle. To avoid imposing a bottleneck,
need to achieve recovery with $\tilde{O}(NK^{2})$ cost. In particular
we can call the matvec oracle only $\tilde{O}(1)$ times.

We will see that this is possible in the case of full permutation
recovery. And again, in the case of partial permutation recovery,
with the same cost we can recover, for a single image, all correspondences
between keypoints in this image and those of all other images.

To fully register all keypoints in all images, adopting the natural
assumption that each of the registry points appears in a fixed fraction
of at least $\tilde{\Omega}(K/M)$ of the images, the approach described
in this section can achieve recovery with $\tilde{O}(N^{2}M)$ additional
cost. Therefore if $M/K=O(1)$, the scaling of this recovery approach
is no worse than that of the optimization approach for (\ref{eq:weakSDPreg}).

\subsubsection{Full permutation}

Recall that in this setting we have $K^{(i)}=K$ for all $i=1,\ldots,N$,
and we seek to recover $X\approx PP^{\top}$, where the blocks $P^{(i)}\in\R^{K\times K}$
of $P$ are \emph{full} permutation matrices.

The key observation is that instead of multiplying $X$ by $E^{(j)}$
as defined by (\ref{eq:Ej0}), it suffices to multiply $X$ by a matrix
with fewer columns. Indeed, within the construction of $E^{(j)}$,
the rows of the identity matrix block can be viewed as `one-hot encodings'\emph{
}of the keypoint indices of the $j$-th image. However, it is possible
to encode these indices with fewer than $K$ bits.

Concretely, we can let 
\begin{equation}
b:[K]\ra\{-1,1\}^{d}\label{eq:binencoding}
\end{equation}
 be defined as follows. Take $d=\lceil\log_{2}K\rceil$, and for $k\in[K]$,
let $b(k)$ be the $d$-digit binary representation of $k-1$, replacing
all $0$'s with $(-1)$'s.

For a more robust encoding, we can let $\tilde{K}>K$ and take $d=\lceil\log_{2}\tilde{K}\rceil$.
Then let $\tau$ be a random permutation on $\tilde{K}$ letters,
and let $b(k)$ be the $d$-digit binary representation of $\tau(k)-1$,
replacing all $0$'s with $(-1)$'s.

Then let 
\[
E^{(j)}=\left(\begin{array}{c}
0_{K\times d}\\
\vdots\\
0_{K\times d}\\
B\\
0_{K\times d}\\
\vdots\\
0_{K\times d}
\end{array}\right)\leftarrow\text{\ensuremath{j}-th block}
\]
 where 
\[
B=\left(\begin{array}{c}
b(1)\\
\vdots\\
b(K)
\end{array}\right)\in\R^{K\times d}
\]
 and the binary encodings $b(k)$ are viewed as $d$-dimensional row
vectors. Importantly, we need only take $d=\tilde{O}(K)$ to obtain
an injective encoding, so $E^{(j)}$ has only $\tilde{O}(K)$ columns.

In the setting of full permutation recovery, we can again assume without
loss of generality that $P^{(1)}=\mathbf{I}_{K}$ and simply consider
$j=1$. Then $X\approx PP^{\top}$ implies that 
\[
Y:=XE^{(1)}\approx\left(\begin{array}{c}
B\\
P^{(2)}B\\
\vdots\\
P^{(N)}B
\end{array}\right).
\]
 Note that $P^{(i)}B$ is a permutation of the rows of $B$. Therefore
$\sigma^{(i)}:[K]\ra[K]$ defined by 
\[
\sigma^{(i)}(k):=\underset{l\in[K]}{\mathrm{argmin}}\ \Vert Y_{k,:}^{(i)}-b(l)\Vert,
\]
 recovers the keypoint correspondence of the $i$-th image to the
first image. In turn we recover $P^{(i)}$ via $P^{(i)}=\delta_{k,\sigma^{(i)}(k)}.$

To maintain injectivity of $\sigma^{(i)}:[K]\ra[K]$, we can compute
$\sigma^{(i)}(k)$ sequentially for $k=1,\ldots,K$, restricting the
$\mathrm{argmin}$ in each step to target indices not yet claimed. 

\subsubsection{Partial permutation}

The binary encoding trick from the last section can be applied in
the setting of partial permutation recovery.

Construct, as in the last section, an injective binary encoding 
\[
b^{(j)}:[K^{(j)}]\ra\{-1,1\}^{d^{(j)}}
\]
 for $j=1,\dots,N$. We can take $d^{(j)}=\lceil\log_{2}K^{(j)}\rceil$,
but we could also take $d^{(j)}$ larger, as described above, for
more robust recovery.

Next we define 
\begin{equation}
E^{(j)}:=\left(\begin{array}{c}
0_{K\times d^{(j)}}\\
\vdots\\
0_{K\times d^{(j)}}\\
B^{(j)}\\
0_{K\times d^{(j)}}\\
\vdots\\
0_{K\times d^{(j)}}
\end{array}\right)\in\R^{L\times d^{(j)}}\label{eq:weakEj}
\end{equation}
 for $j=1,\ldots,N$, where 
\[
B^{(j)}:=\left(\begin{array}{c}
b^{(j)}(1)\\
\vdots\\
b^{(j)}(K)
\end{array}\right)\in\R^{K^{(j)}\times d^{(j)}}.
\]
 Then the recovery procedure is presented in Algorithm \ref{alg:weakpartialrecovery},
which only makes small suitable modifications to Algorithm \ref{alg:strongpartialrecovery}.

\begin{algorithm}
\caption{Fast partial permutation recovery}

\begin{algorithmic}[1]

\Require{Matvec oracle for the primal solution $X$, injective binary encodings $b^{(i)}:[K^{(i)}] \ra \{-1,1\}^{d^{(i)}}$ for $i=1,\ldots,N$}
        
\State{Initialize $m=0$ and $\mathcal{S}^{(i)}=[K^{(i)}]$ for $i=1,\ldots,N$}

\While{there exists $i\in[N]$ such that $\mathcal{S}^{(i)}\neq\emptyset$}

\State{Set $j\leftarrow\underset{j'\in[N]}{\mathrm{argmin}}\left\{ \sum_{i=1}^{N}\Vert Q_{\mathcal{S}^{(i)},\mathcal{S}^{(j')}}^{(i,j')}\Vert_{\mathrm{F}}^{2}\right\}$, or alternatively choose uniformly randomly from indices not previously chosen}

\State{Form $E^{(j)}\in\R^{L\times d^{(j)}}$ following (\ref{eq:weakEj}), using the binary encoding $b^{(j)}$}

\State{Form $Y=XE^{(j)}$ using $d^{(j)}$ matvecs by $X$}

\For{$l\in\mathcal{S}^{(j)}$}

\State{Set $m\leftarrow m+1$}
\State{Set $\mathcal{R}(j,l)\leftarrow m$}

\EndFor

\For{$i\in[N]\backslash\{j\}$}

\State{Set $\mathcal{T} \leftarrow \mathcal{S}^{(j)}$}

\For{$k\in\mathcal{S}^{(i)}$}

\State{Set $l\leftarrow\underset{l\in\{0\}\cup \mathcal{T} }{\mathrm{argmin}}\left\{ \Vert Y_{k,:}^{(i)}-b^{(j)}(l) \Vert\right\}$}

\If{$l\neq 0$}
\State{Set $\mathcal{S}^{(i)}\leftarrow\mathcal{S}^{(i)}\backslash\{k\}$}
\State{Set $\mathcal{T} \leftarrow \mathcal{T} \backslash \{l\}$}
\State{Set $\mathcal{R}(i,k)\leftarrow\mathcal{R}(j,l)$}
\EndIf

\EndFor
\EndFor

\State{$\mathcal{S}^{(j)}\leftarrow\emptyset$}

\EndWhile

\State{Set $M \leftarrow m$}

\State{\Return{ $\mathcal{R}:\mathcal{K}\ra [M]$ }}
\end{algorithmic}\label{alg:weakpartialrecovery}
\end{algorithm}

By the same arguments as in the end of Section \ref{sec:partialrecov},
the algorithm (with random selection of the index $j$ at each stage)
terminates with high probability in $f^{-1}\log(M/\delta)$ steps,
under the assumption that each registry point appears in a fraction
of at least $f$ of the images. Thus the expected total cost of the
recovery is $\tilde{O}(N^{2}K/f)$. Under the natural assumption that
$f=\tilde{O}(K/M)$, we obtain a recovery cost of $\tilde{O}(N^{2}M)$.

\subsection{Masked recovery}\label{sec:mask}

Finally, we present a simpler recovery procedure in which we classify every observed correspondence as correct or incorrect and in particular do not insist on cycle consistency.

The procedure is simple and relies on using the stochastic Cholesky approximation $$ X \approx \hat{X} := \frac{1}{S} [X^{1/2} Z] \, [X^{1/2} Z]^\top $$ where $Z \in \mathbb{R}^{NK \times S}$ is a random matrix with independent standard Gaussian entries, to approximate the entries of $X$ on the nonzero pattern of the given correspondence matrix $Q$. Note that this is the same stochastic approximation employed within Algorithms~\ref{alg:strong} and~\ref{alg:weak}. After obtaining $\hat X \leftarrow Q \odot \hat{X}$, we round it to a binary matrix by thresholding its entries. To find the cut-off threshold, we provide two options.

\noindent \textbf{Option 1:} We first fit a two-component Gaussian Mixture Model (GMM) to the distribution of nonzero entries of $\hat X$ using the Expectation-Maximization algorithm, sorting the components by their means. The cutoff is then computed as the intersection point of the two Gaussian probability density functions, derived by solving for the value where their likelihoods are equal. 

\noindent\textbf{Option 2:} In cases where the distribution of the entries in $\hat{X}$ is not clearly bimodal, the threshold is set to the 80th or 90th percentile of the entries.

The pseudocode for the masked recovery procedure is provided in Algorithm~\ref{alg:maskedrecovery}.

% Matvec oracle $X^{1/2}$

% $Z \in \mathbb{R}^{NK \times S}$

% $W = X^{1/2} Z \in \mathbb{R}^{NK \times S}$

% $W_k^{(i)} \in \mathbb{R}^{S}$

% score $s_{kl}^{(i,j)} := \frac{1}{S} \,  W_k^{(i)} \cdot W_l^{(j)}$

% For $(i,j,k,l)$ such that $Q_{kl}^{(i,j)} \neq 0$,

\begin{algorithm}
\caption{Masked recovery}

\begin{algorithmic}[1]

\Require{Matvec oracle for the square-root primal solution $X^{1/2}$, number $S$ of stochastic shots}

\State{Let $Z \in \mathbb{R}^{L \times S}$ have i.i.d. standard Gaussian entries}

\State{Form $W = X^{1/2} Z \in \mathbb{R}^{L \times S}$, denote its rows as $w_k^{(i)} \in \mathbb{R}^{S}$ }

\State $\hat{X} \leftarrow 0 \in \R^{L\times L}$ (sparse)

\For{ $i,j,k,l$ such that $Q_{kl}^{(i,j)} \neq 0$}

\State{Set $\hat{X}_{kl}^{(i,j)} \leftarrow \frac{1}{S} \,  w_k^{(i)} \cdot w_l^{(j)}$ }

\EndFor

\State{$\hat X \leftarrow \text{threshold} (\hat X)$}

\State{\Return{ $\hat X$ }}
\end{algorithmic}\label{alg:maskedrecovery}
\end{algorithm}

\section{Experiments}\label{sec:experiments}
We compare the accuracy and speed of different algorithms on both synthetic and real datasets. 
\subsection{Implementation Details}
All experiments were conducted on a MacBook Pro equipped with an Apple M1 Pro chip and 16 GB of memory.

\vspace{2mm}
\noindent\textbf{Proposed Methods:}

\noindent We evaluate the performance of our proposed algorithms, denoted as $$\texttt{SDP-(weak/strong)-(fast/slow/thresh)}$$ according to which SDP relaxation is optimized and which recovery procedure is used after optimization. Here, \textit{weak} and \textit{strong} refer to the corresponding problem formulations \eqref{eq:weakSDPreg} and \eqref{eq:strongSDPreg}; \textit{fast} and \textit{slow} refer to the cycle-consistent recovery strategies of Algorithms \ref{alg:weakpartialrecovery} and \ref{alg:strongpartialrecovery}, while \textit{thresh} refers to the masked recovery strategy with threshold-based rounding (Algorithm \ref{alg:maskedrecovery}).

For all our methods, the damping parameter is set to $\gamma = 5$. The number of random vectors is $S = 20$ for the weak formulation and $S = 20 \,\max_i \{ K^{(i)} \}$ for the strong formulation. Both formulations use the same regularization parameter, chosen as $\beta = \lambda \log N / N$, where $\lambda \in \{5, 10, 20\}$. While larger values of $\lambda$ (and hence $\beta$) typically yield better accuracy at higher computational cost, we observe that in real data scenarios, which are more challenging and noisy, increasing $\lambda$ does not always improve performance. Therefore, we fix $\lambda = 5$ for all real-data experiments.

For the fast recovery method, we set $\tilde{K} = 10 \, \max_i \{ K^{(i)} \}$ in all experiments. The maximum number of iterations is set to 20 for \texttt{SDP-weak} and 10 for \texttt{SDP-strong}.

\vspace{2mm}
\noindent \textbf{Previous methods:}

\noindent  For previous methods, we compare with MatchEIG~\cite{MatchEIG}, Spectral~\cite{deepti}, and MatchFAME~\cite{MatchFAME}. Both MatchEIG and Spectral recover the top eigenvectors of the matrix $ Q $, but differ in their rounding procedures. In Spectral, each $ P^{(i)} $ is recovered by applying the Hungarian algorithm to the $ i $-th block of the eigenmatrix formed by the top eigenvectors. In contrast, MatchEIG does not recover the individual $ P^{(i)} $, but instead directly approximates the ground-truth matrix $ Q $. Specifically, each block $ Q^{(i,j)} $ is rounded to a binary matrix using a more heuristic but computationally faster method than the Hungarian algorithm.

The current state of the art MatchFAME combines a weighted version of Spectral with the inexpensive rounding approach used in MatchEIG to initialize $ P $. It then refines the solution using the iteratively reweighted projected power method (PPM)~\cite{Chen_PPM}, where the weights are estimated via a message-passing algorithm based on cycle consistency~\cite{cemp, MatchFAME}.

We use the original code and their default choices to implement the above methods.

\vspace{2mm}
\noindent\textbf{Evaluation metrics:}

\noindent For practical applications, we evaluate the performance of the PPS methods in classifying correct and incorrect keypoint matches. Given a matrix $ Q $ that encodes candidate keypoint correspondences, the goal is to refine $ Q $ by removing incorrect nonzero entries corresponding to false matches. Ideally, we aim to recover the matrix $ Q \odot Q^* $, where $ Q^* $ denotes the ground-truth match matrix and $ \odot $ represents the elementwise product. In practice, this is approximated by $\hat Z= Q \odot \hat{X} $, where $ \hat{X} $ is the output of a given method. We refer to $\hat Z$ as our filtered match, which retains only a subset of the existing matches. We note that for the masked recovery $\hat X=\hat Z$.

To evaluate accuracy, we use the following three metrics:
\begin{itemize}
    \item \textbf{Precision} (higher is better): the proportion of true matches among the retained matches.
    \[
    \text{Precision} = \frac{\text{nnz}(\hat{Z} \odot Q^*)}{\text{nnz}(\hat{Z})}.
    \]
    
    \item \textbf{Recall} (higher is better): the proportion of true matches retained after filtering, relative to the total number of ground-truth matches.
    \[
    \text{Recall} = \frac{\text{nnz}(\hat{Z} \odot Q^*)}{\text{nnz}(Q^*)}.
    \]
    
    \item \textbf{F1-Score} (higher is better): the harmonic mean of precision and recall, providing a balanced measure of performance
    \[
    \text{F1-Score} = \frac{2}{\frac{1}{\text{Precision}} + \frac{1}{\text{Recall}}}.
    \]
\end{itemize}

We note that while the F1-score provides a balanced evaluation by equally weighting precision and recall, in practice, precision is often considered more important. This is because it is critical that the remaining keypoint matches are correct, as long as their number is not too small.

\subsection{Synthetic Data}
We first compare the accuracy and speed of the different algorithms on synthetic data.

Specifically, we generate the data using the following model. We first set $N=100$ and $M=1000$. For each pair $(i, j)$, we define the ground-truth partial permutation matrix as
\[
Q^{(i,j)} = P^{(i)}P^{(j)\top},
\]
where each $P^{(i)} \in \mathbb{R}^{K^{(i)} \times M}$ is a partial permutation matrix. The value $K^{(i)}$ is drawn independently from the uniform distribution on the interval $[100, 200]$, and each $P^{(i)}$ is sampled uniformly from the space of partial permutations of the given size, with no all-zero rows.

We then corrupt each $Q^{(i,j)}$ independently with probability $q$, referred to as the corruption probability. Specifically, for each corrupted pair $(i, j)$, we generate two random partial permutation matrices $\widetilde{P}_1, \widetilde{P}_2 \in \mathbb{R}^{K^{(i)} \times M}$, each sampled uniformly from the space of partial permutations of the same size as $P^{(i)}$, with no all-zero rows. The corrupted observation is then given by
\[
\widetilde{Q}^{(i,j)} = \widetilde{P}_1 \widetilde{P}_2^\top.
\]
We note that for each corrupted pair $(i,j)$ one needs to resample $\widetilde{P}_1, \widetilde{P}_2$.

\begin{figure}[H]
    \centering    \includegraphics[width=1\linewidth]{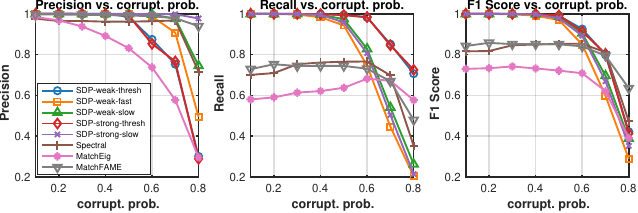}
    \caption{Performance comparison of algorithms. For both weak and strong formulations, $\lambda=20$.}
    \label{fig:line}
\end{figure}
In Figure~\ref{fig:line}, we compare the precision, recall, and F\textsubscript{1} score of our proposed methods against existing approaches. For the masked recovery with thresholding, the number of random samples is set to 200. Note that this value is independent of the parameter $S$ used for optimizing the SDP formulation. To round the solution to a binary matrix, we adopt the thresholding strategy described as Option 1 in Section~\ref{sec:mask}. For each corruption probability, all methods are run 10 times, and the reported precision, recall, and F\textsubscript{1} scores are averaged over these trials.

We observe that the \texttt{SDP-strong} formulation with slow recovery achieves the highest precision, comparable to that of MatchFAME. However, MatchFAME often yields significantly lower recall. Our SDP formulations with masked recovery achieve the best F\textsubscript{1} score across most corruption levels. In general, the proposed SDP-based methods substantially outperform both Spectral and MatchEIG, which may be attributed to their inaccurate estimates of $M$ (noting that $2K \neq M$). In contrast, our methods do not require any prior estimation of $M$.

\begin{figure}[H]
    \centering    \includegraphics[width=1\linewidth]{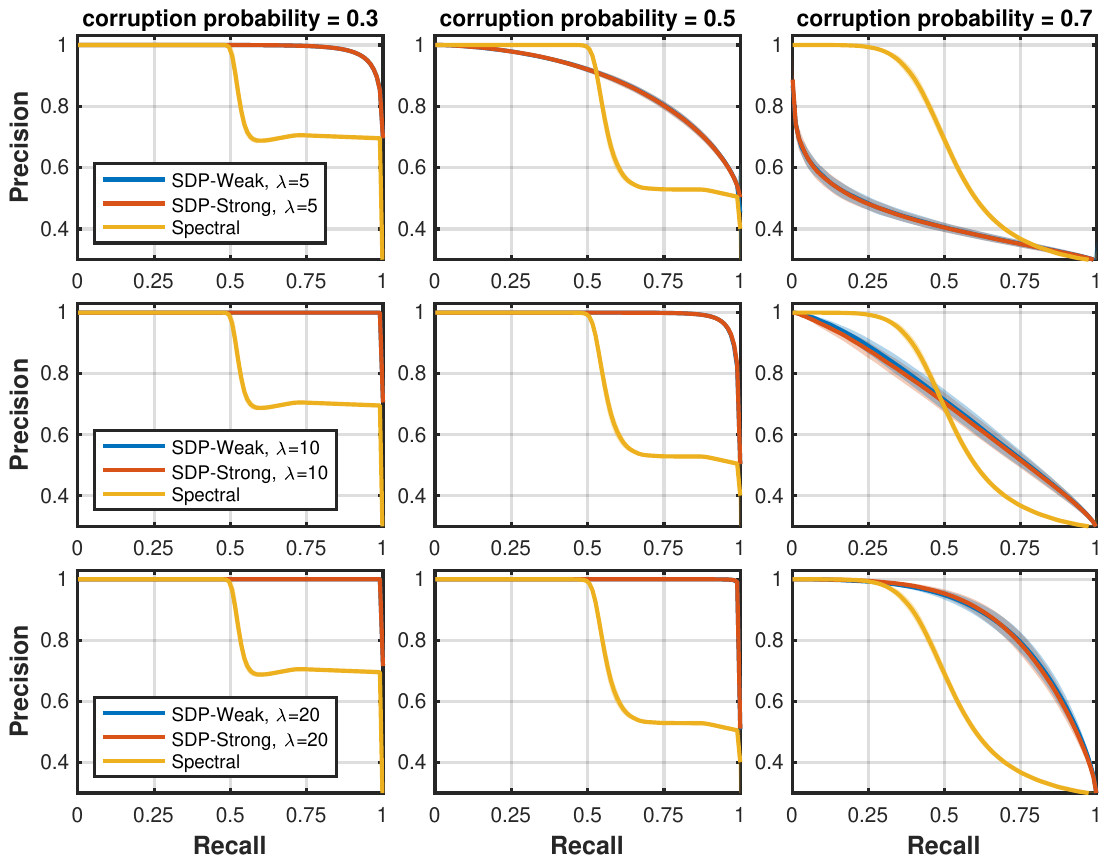}
    \caption{Precision-Recall curve for SDP and Spectral. Every point on the curve corresponds to a different threshold for the masked recovery. Each column of plots corresponds to different corruption probability. Each row corresponds to a different choice of $\lambda$ for the SDP formulation.}
    \label{fig:PR}
\end{figure}

In Figure~\ref{fig:PR}, we plot the precision-recall curves for the refined matrices $\hat{X}$ obtained from our SDP formulations and the low-rank approximations produced by spectral methods. For both the weak and strong SDP formulations, we fix the regularization parameter $\lambda$ and apply masked recovery with varying thresholds. The number of random samples for masked recovery is set to 200. Each threshold value yields a pair of precision and recall values, which corresponds to a point on the curve. Varying the threshold generates a full precision-recall curve for each method. For each threshold, we perform 10 independent trials, allowing us to compute and visualize a confidence region around each curve, defined as $\pm$ one standard deviation. Curves that lie closer to the top-right corner indicate better overall performance.

For the Spectral method, we construct the refined matrix as $VV^\top$, where $V$ consists of the top $\hat{M} = 2K$ eigenvectors of $Q$. Instead of rounding $VV^\top$ using the MatchEIG procedure or the Hungarian algorithm, we apply direct hard thresholding. Varying the threshold produces a corresponding precision-recall curve for the Spectral method.

We observe that increasing $\lambda$ generally improves the performance of the SDP-based methods. The weak and strong formulations perform almost the same across all settings. Our SDP approaches outperform the Spectral method in most scenarios, particularly when the corruption probability is less than or equal to 0.5. Even at a corruption level of 0.7, our SDP method with $\lambda = 20$ still demonstrates superior performance over Spectral.

\begin{table}[H]
  \centering
  \footnotesize                        % slightly shrink the font
  \setlength{\tabcolsep}{3pt}          % tighten column padding
  \caption{Runtime (seconds) for each method and setting.}
  \label{tab:runtime}
  %
  % \resizebox will scale the entire table to \textwidth
  \resizebox{\textwidth}{!}{%
    \begin{tabular}{@{}cc || cc | ccc |ccc| cc@{}}
      \toprule
      \multicolumn{2}{c}{Data Size} &
      \multicolumn{2}{c}{MatchEIG} &
      \multicolumn{3}{c}{SDP-Weak} &
      \multicolumn{3}{c}{SDP-Strong} &
      \multicolumn{2}{c}{Recovery} \\
      \cmidrule(lr){1-2}\cmidrule(lr){3-4}\cmidrule(lr){5-7}\cmidrule(lr){8-10}\cmidrule(lr){11-12}
        $N$ & $K$ & Eig.\ & Round &
        $\lambda\!=\!5$ & $\lambda\!=\!10$ & $\lambda\!=\!20$ &
        $\lambda\!=\!5$ & $\lambda\!=\!10$ & $\lambda\!=\!20$ &
        Weak & Strong \\
      \midrule
      % --------- example rows ---------
      100  & 10 & 0.09  & 0.77  & 0.18 & 0.18 & 0.21  & 0.29  & 0.29 & 0.54  &0.27  & 0.28  \\
      100 & 100 & 2.72 & 6.56 & 1.29 & 1.58 & 1.72 & 14.2 &20.0 & 29.4 & 0.29 &  11.3\\
      100 & 1000 & 971.5 & 531.3 & 52.5 & 54.6 & 83.1
      & --- & --- &---& 3.03 & 122.3\\
      \hline 
      10 & 100 & 0.40 & 0.10 & 0.16 & 0.17 & 1.50 & 1.33 & 1.57 & 1.98 & 0.03 & 0.17\\
      100 & 100 & 2.72 & 6.56 & 1.29 & 1.58 & 1.72 & 14.2 &20.0 & 29.4 & 0.29 &  11.3\\
      1000 & 100 & 19.5 & 2001& 52.7 & 52.9 & 70.2 & ---&---&--- & 1119 & 3079 \\
      \midrule
      \bottomrule
    \end{tabular}%
  }
\end{table}

In Table~\ref{tab:runtime}, we compare the runtime (in seconds) of different methods across varying values of $N$ and $K$. We observe that the \texttt{SDP-weak} formulation offers a significant runtime advantage over MatchEIG when $N$ is relatively small and $K$ is large (e.g., the row with $N = 100$, $K = 1000$), which aligns with typical settings in real-world datasets. In many structure-from-motion (SfM) scenarios, the number of cameras ($N$) ranges from 10 to 500, while the number of keypoints per image ($K$) is often larger. These results suggest that \texttt{SDP-weak} scales more effectively in practice compared to spectral methods.

We also note that the \texttt{SDP-strong} formulation is slower than both MatchEIG and \texttt{SDP-weak}. For larger values such as $N = 1000$ or $K = 1000$, the \texttt{SDP-strong} method exceeds the memory capacity of the testing machine.

\subsection{Real Data}
To assess the practical effectiveness of our approach, we conducted experiments on the EPFL multi-view stereo benchmark \cite{EPFL} -- an established collection of real-world outdoor scenes widely used for evaluating keypoint matching algorithms. The dataset consists of six image sequences, each containing 8 to 30 photographs captured from different viewpoints.
Feature extraction was performed using the SIFT algorithm \cite{sift04}. Candidate correspondences between image pairs were established through nearest-neighbor search with a ratio test, and further filtered using geometric verification via RANSAC \cite{RANSAC:81}. This preprocessing pipeline produced the initial keypoint matches for all tested methods.

Since the ground-truth set of keypoints shared across views (i.e., the set of 3D points) is not directly available, we followed the strategy proposed in~\cite{MatchEIG} to estimate its size. Specifically, we set the universe size to twice the average number of keypoints per image and supplied this estimate to all baseline methods. Notably, our SDP-based algorithms operate independently of this parameter and do not rely on any universe size estimate.

A match is considered correct if the predicted correspondence lies within a specified distance threshold from the ground-truth. Since the datasets provide ground-truth camera parameters, we can predict the epipolar line on which a correct match should lie. The distance threshold is set to $1\%$ of the image diagonal, following \cite{MatchEIG}.

For all of our SDP formulations, we use $\lambda = 5$. Given the noisy nature of the dataset, we set the number of random samples for masked recovery to $S = 1000$. However, the distribution of entries in $\hat{X}$ does not exhibit a clearly bimodal pattern. As a result, we do not apply a Gaussian mixture model to estimate the threshold. Instead, we adopt the heuristic described as Option 2 in Section~\ref{sec:mask}, selecting the threshold as the 90th percentile of the entries in $\hat{X}$. In other words, we remove the top 10\% of nonzero entries, corresponding to the most uncertain matches.

While MatchFAME achieves higher precision in many cases, it frequently suffers from very low recall. In contrast, our SDP-based methods strike a better balance between precision and recall, leading to superior F\textsubscript{1} scores. Moreover, our weak formulation, when paired with either fast or masked recovery, is typically faster than existing approaches.

\begin{table}[t]
\centering
\resizebox{\linewidth}{!}{
\renewcommand{\arraystretch}{1}
\tabcolsep=0.08cm
\begin{tabular}{|l||cccc|cccc|cccc|cccc|cccc|cccc|}
\hline
\textbf{Dataset} & \multicolumn{4}{c|}{\textbf{Herz-Jesu-P8}} &
   \multicolumn{4}{c|}{\textbf{Herz-Jesu-P25}} &
   \multicolumn{4}{c|}{\textbf{Fountain-P11}} &
   \multicolumn{4}{c|}{\textbf{Entry-P10}} &
   \multicolumn{4}{c|}{\textbf{Castle-P19}} &
   \multicolumn{4}{c|}{\textbf{Castle-P30}} \\
\hline
$n$     & \multicolumn{4}{c|}{8}   & \multicolumn{4}{c|}{25}  & \multicolumn{4}{c|}{11}  & \multicolumn{4}{c|}{10}  & \multicolumn{4}{c|}{19}  & \multicolumn{4}{c|}{30}  \\
$\hat M$ & \multicolumn{4}{c|}{386} & \multicolumn{4}{c|}{517} & \multicolumn{4}{c|}{374} & \multicolumn{4}{c|}{432} & \multicolumn{4}{c|}{314} & \multicolumn{4}{c|}{445} \\
\hline
\textbf{Method $\backslash$ Stat.} & P & R & F & T & P & R & F & T & P & R & F & T & P & R & F & T & P & R & F & T & P & R & F & T \\
\hline
Input       & 94   & --  & --  & --  & 89.6 & --  & --  & --  & 94.3 & --  & --  & --  & 75.8 & --  & --  & --  & 69.7 & --  & --  & --  & 71.8 & --  & --  & --  \\\hline
MatchEIG    & 95.4 & 92.9 & 94.1 &  2.3 & 94.4 & 76.5 & 84.5 & 46   & 95.8 & 89.7 & 92.6 & 16.1 & 81.4 & 90.3 & 85.6 & 24.1 & 79.1 & 65.3 & 71.5 & 23.6 & 84.1 & 64.6 & 73.1 & 34.2\\\hline
Spectral    & 94.9 & 82.9 & 88.5 &  4.4 & 92.3 & 84   & 88   & 72.2 & 95.6 & 92.2 & 93.9 & 13.8 & 82.6 & 84.6 & 83.6 & 19.4 & 76.8 & 83.9 & 80.2 & 26.5 & 80.3 & 80.5 & 80.4 & 94.5\\\hline
MatchFame   & 95.9 & 89.7 & 92.7 &  2.6 & 95.7 & 79.8 & 87   & 13.8 & 97.2 & 90.9 & 93.9 &  3.4 & 86.7 & 74.5 & 80.1 &  4.1 & 88.6 & 57.9 & 70   &  5.5 & 90.9 & 65.3 & 76   & 14.7\\\hline
SDP-weak-fast   & 95.6 & 92.5 & 94   &  2.2 & 93.8 & 87   & 90.3 &  5.6 & 96   & 91.3 & 93.6 &  2.8 & 82.6 & 78.6 & 80.5 &  2.9 & 79.8 & 75.8 & 77.8 &  3.6 & 83.6 & 75.3 & 79.2 &  6.1\\\hline
SDP-weak-slow   & 95.5 & 96.3 & 95.9 &  3   & 93.7 & 88.5 & 91.1 & 22.2 & 95.7 & 93.6 & 94.7 &  4.4 & 81.9 & 85.1 & 83.5 &  5   & 80.3 & 82.9 & 81.6 &  8   & 83.5 & 80   & 81.7 & 25.5\\\hline
SDP-weak-thresh & 95.9 & 91.8 & 93.8 & 1.3 & 93.3 & 93.5 & 93.4 & 5.3 & 96.0 & 91.7 & 93.8 & 2.7 & 78.6 & 93.0 & 85.2 & 3.1 & 75.1 & 96.5 & 84.5 & 3.4 & 77.2 & 96.7 & 85.8 & 5.0\\\hline
SDP-strong-slow & 95.4 & 96.9 & 96.1 & 36.6 & 94.2 & 87.7 & 90.8 &540.9 & 95.9 & 94.5 & 95.2 & 49.5 & 82.3 & 90   & 86   & 72   & 80.5 & 82.7 & 81.6 & 76.8 & 84.7 & 80.4 & 82.5 & 371 \\\hline
SDP-strong-thresh & 96.3 & 92.1 & 94.1 & 35.3 & 93.8 & 94.0 & 93.9 & 527.1
 & 96.7 & 92.3 & 94.4 & 49.3 & 80.6 & 95.4 & 87.4 & 73.5 & 75.4 & 97.0 & 84.9 & 74.4 & 77.6 & 97.2 & 86.3 & 332.3\\\hline
\end{tabular}}
\caption{Performance on the EPFL datasets. $n$ is the number of cameras; $\hat M$, the approximated $M$, is twice the averaged $K_i$ over $i\in [n]$; P is the precision rate (the percentage of good matches within the filtered match); R is the recall rate (the percentage of remaining good matches after filtering compared to the total good matches before filtering); F is the F1-score $2/(1/P+1/R)$; T is runtime in seconds.}
\label{tab:epfl_final}
\end{table}

\appendix

\section{Proof of Theorem \ref{thm:reg} \label{sec:thm1}}
\begin{proof}
[Proof of Theorem \ref{thm:reg}.] It suffices to check the claims
for the weakest relaxation, i.e., the GW-type relaxation (\ref{eq:weakestSDP}).

Importantly, the constraints $X\succeq0$ and $\mathrm{diag}(X)=\mathbf{1}_{L}$
alone imply that $\vert X_{kl}^{(i,j)}\vert\leq1$ for all $i,j,k,l$.
(This follows directly from the fact that the determinant of any principal
$2\times2$ submatrix must be nonnegative.)

Since the entries of $Q$ are all either $0$ or $1$, for any feasible
$X$ it follows that 
\begin{equation}
\Tr[QX]\leq\mathbf{1}_{L}^{\top}Q\mathbf{1}_{L}.\label{eq:QXless}
\end{equation}
 But also 
\[
\Tr[QX^{\star}]=\Tr[Q^{2}]=\Vert Q\Vert_{\mathrm{F}}^{2}=\mathbf{1}_{L}^{\top}Q\mathbf{1}_{L}.
\]
 The last equality follows from the fact that since the entries of
$Q$ are all either 0 or 1, the Frobenius norm of $Q$ coincides with
the sum of its entries. Therefore $X^{\star}$ is an optimal solution.
This establishes the first claim of the theorem.

Note that the inequality (\ref{eq:QXless}) holds with equality for
feasible $X$ if and only if the entries of $X$ on the nonzero sparsity
pattern of $Q$ are all 1's. In other words, any feasible $X$ is
optimal if and only if the entries of $X$ on the nonzero sparsity
pattern of $Q$ are all 1's.

Before proving the second claim, we explain how we can view $Q$ as
a block-diagonal matrix by suitable reindexing. Let $(i,k)$ be a
multi-index for the rows of $Q$, with $i\in\{1,\ldots,N\}$ denoting
the block index and $k\in\{1,\ldots,K^{(i)}\}$ denoting the index
within the $i$-th block. Then we can partition the row indices into
$M$ disjoint subsets
\[
\mathcal{I}_{m}=\{(i,k)\,:\,P_{km}^{(i)}=1\},\quad m=1,\ldots,M.
\]
 Here $\mathcal{I}_{m}$ can be viewed as the set of keypoints (across
all images) that correspond to the $m$-th registry point. Evidently
each keypoint must fall into exactly one of these subsets. Note that
$Q_{k,l}^{(i,j)}=1$ if and only if $(i,k)$ and $(j,l)$ lie in the
same $\mathcal{I}_{m}$. Let $L_{m}=\vert\mathcal{I}_{m}\vert$.

Therefore we can view $Q$ as a direct sum or block-diagonal matrix
according to this partition of the rows, i.e., as 
\begin{equation}
Q=\bigoplus_{m=1}^{M}\mathbf{1}_{L_{m}}\mathbf{1}_{L_{m}}^{\top}.\label{eq:Qblockdiag}
\end{equation}
 After suitable reindexing the diagonal blocks are simply $L_{m}\times L_{m}$
matrices consisting of all 1's.

Then it is clear that $X$ agrees with $Q$ on the sparsity pattern
of $Q$ if and only if $X_{\mathcal{I}_{m},\mathcal{I}_{m}}=\mathbf{1}_{L_{m}}\mathbf{1}_{L_{m}}^{\top}$
for all $m=1,\ldots,M$. We already established that such agreement
is equivalent to optimality for feasible $X$.

Therefore, to prove the second claim in the theorem, consider the
following problem: 
\begin{align}
\underset{X\in\R^{L\times L}}{\text{minimize}}\ \  & S(X)\label{eq:maxent}\\
\text{subject to}\ \  & X_{\mathcal{I}_{m},\mathcal{I}_{m}}=\mathbf{1}_{L_{m}}\mathbf{1}_{L_{m}}^{\top},\quad m=1,\ldots,M,\nonumber \\
 & X\succeq0.\nonumber 
\end{align}
 It suffices to show that the unique solution is $X^{\star}$. (Since
$S$ is strictly concave, uniqueness is automatically guaranteed.)

By the quantum Gibbs variational principle~\cite{Dissertation}, which tells
us how to minimize entropy subject to linear constraints, intuitively
we expect that the solution $\tilde{X}$ of (\ref{eq:maxent}) must
take the form 
\[
\tilde{X}=e^{\bigoplus_{m=1}^{M}A_{m}}=\bigoplus_{m=1}^{M}e^{A_{m}},
\]
 where the $A_{m}\in\R^{M_{m}\times M_{m}}$ are symmetric matrices
which are dual variables for the block constraints of (\ref{eq:maxent}).
However, to interpret such a claim appropriately we are forced to
consider $A_{m}$ with some negatively infinite eigenvalues, because
the constraints of (\ref{eq:maxent}) force any feasible $X$ to have
a nontrivial null space. Allowing ourselves (for the moment) to stretch
the interpretation of $A_{m}$ accordingly, the constraints of (\ref{eq:maxent})
then imply that $e^{A_{m}}=\mathbf{1}_{M_{m}}\mathbf{1}_{M_{m}}^{\top}$,
and therefore by (\ref{eq:Qblockdiag}), $\tilde{X}=Q=X^{\star}$,
as was to be shown.

To complete the proof rigorously, we can first `quotient out' by the
null space that is demanded by the constraints, then apply the Gibbs
variational principle.

To this end, for each $m$, let $\{u_{mn}\}_{n=1}^{L_{m}-1}$ be a
orthonormal spanning set for the orthogonal complement of $\mathbf{1}_{L_{m}}$
in $\R^{L_{m}}$. Then for any $X$ that is feasible for (\ref{eq:maxent}),
the vectors 
\[
v_{mn}:=0\oplus\cdots\oplus\,0\,\oplus\,\underbrace{u_{mn}}_{\text{\ensuremath{m}-th slot}}\,\oplus\,0\,\oplus\cdots\oplus\,0\in\R^{L}
\]
must satisfy $v_{mn}^{\top}Xv_{mn}=0$, and in turn it follows (since
$X$ is positive semidefinite) that $Xv_{mn}=0$, i.e., $v_{mn}$
are null vectors for $X$.

Meanwhile, let 
\[
w_{m}:=0\,\oplus\cdots\oplus\,0\,\oplus\,\underbrace{\mathbf{1}_{m}/\sqrt{L_{m}}}_{\text{\ensuremath{m}-th slot}}\,\oplus\,0\,\oplus\cdots\oplus0\in\R^{L}.
\]
 Then the set $\{w_{m}\}$ completes $\{u_{mn}\}$ to an orthonormal
basis of $\R^{L}$, and the matrix of $X$ in the basis $\{w_{m},u_{mn}\}$,
where the $w_{m}$ are ordered to appear first, is 
\[
X=\left(\begin{array}{cc}
Y & 0\\
0 & 0
\end{array}\right),
\]
 where $Y_{mn}=w_{m}^{\top}Xw_{n}$. For such $X$, the block-diagonal
constraint of (\ref{eq:maxent}) is equivalent to having $Y_{mm}=w_{m}^{\top}Xw_{m}=L_{m}$.

Moreover, $S(X)=S(Y)$, so we can rephrase the problem (\ref{eq:maxent})
equivalently as: 
\begin{align*}
\underset{X\in\R^{M\times M}}{\text{minimize}}\ \  & S(Y)\\
\text{subject to}\ \  & Y_{mm}=L_{m},\quad m=1,\ldots,M,\\
 & Y\succeq0.
\end{align*}
 Then the Gibbs variational principle establishes that the unique
optimizer is 
\[
\tilde{Y}=e^{\mathrm{diag}(\lambda)}
\]
 for a suitable dual variable $\lambda\in\R^{M}$. In particular,
$\tilde{Y}$ is a diagonal matrix with diagonal $\tilde{Y}_{mm}=L_{m}$.

Then we can recover 
\[
\tilde{X}=\sum_{m=1}^{M}\tilde{Y}_{mn}w_{m}w_{n}^{\top}=\sum_{m=1}^{M}L_{m}w_{m}w_{m}^{\top}=\bigoplus_{m=1}^{M}\mathbf{1}_{L_{m}}\mathbf{1}_{L_{m}}^{\top},
\]
 i.e., $\tilde{X}=Q=X^{\star}$, as desired.
\end{proof}

\section{Proof of Theorem \ref{thm:reg2} \label{sec:thm2}}
\begin{proof}
[Proof of Theorem \ref{thm:reg2}.] In this proof we will make use
of the same reindexing that renders $Q$ as a block-diagonal matrix,
following (\ref{eq:Qblockdiag}) in the proof of Theorem \ref{thm:reg}.

First we will prove that the optimizer $X_{\beta}$ must be block-diagonal
after this reindexing. To show this, let $X_{\beta}'$ be the matrix
obtained from $X_{\beta}$ by zeroing all entries off the block diagonal.
Then our aim is to establish that $X_{\beta}'=X_{\beta}$.

Note that $\Tr[QX_{\beta}]=\Tr[QX_{\beta}']$, i.e., the linear objective
is unchanged by this replacement, due to the block-diagonality (\ref{eq:Qblockdiag})
of $Q$. Moreover, we claim that $S(X_{\beta}')\leq S(X_{\beta})$.
To see this, we will show that $X_{\beta}'$ is in fact the optimizer
of 
\begin{align}
\underset{X\in\R^{L\times L}}{\text{minimize}}\ \  & S(X)\label{eq:maxent-1}\\
\text{subject to}\ \  & X_{\mathcal{I}_{m},\mathcal{I}_{m}}=[X_{\beta}]_{\mathcal{I}_{m},\mathcal{I}_{m}},\quad m=1,\ldots,M,\nonumber \\
 & X\succeq0.\nonumber 
\end{align}
 Indeed, as in the proof of Theorem \ref{thm:reg}, the Gibbs variational
principle\footnote{If $X_{\beta}$ has a null space, we can quotient out by the null
space in the same fashion as in the proof of Theorem \ref{thm:reg},
then apply the Gibbs variational principle. In fact we will see concretely
later that $X_{\beta}$ is never singular for finite $\beta$.} implies that the the optimizer is block-diagonal, and since the diagonal
blocks are constrained to agree with those of $X_{\beta}$, the optimizer
is precisely $X_{\beta}'$. This implies that $S(X_{\beta}')\leq S(X_{\beta})$.

Finally, we claim that for all three regularized problems (\ref{eq:strongSDPreg}),
(\ref{eq:weakSDPreg}), and (\ref{eq:weakestSDPreg}), the substitution
$X_{\beta}\ra X_{\beta}'$ preserves feasibility. Indeed, the (non-reindexed)
diagonal blocks satisfy $\left[X_{\beta}'\right]^{(i,i)}=\mathbf{I}_{K^{(i)}}$
for all $i$, because $\mathrm{diag}(X_{\beta})=\mathrm{diag}(X_{\beta}')=\mathbf{1}_{L}$
and moreover $(i,k)$ and $(i,l)$ lie in the same set $\mathcal{I}_{m}$
only if $k=l$ (i.e., distinct keypoints in the same image correspond
to distinct registry points), so all the non-diagonal entries of $\left[X_{\beta}'\right]^{(i,i)}$
are zeroed out. This implies the claim.

In summary, $X_{\beta}'$ is feasible for (\ref{eq:strongSDPreg}),
(\ref{eq:weakSDPreg}), and (\ref{eq:weakestSDPreg}), and it improves
the value of the objective relative to the optimizer $X_{\beta}$.
But the optimizer is unique by strict convexity, so $X_{\beta}=X_{\beta}'$.

Then we can add the constraint that $X$ is block-diagonal (after
reindexing) to any of the regularized problems (\ref{eq:strongSDPreg}),
(\ref{eq:weakSDPreg}), and (\ref{eq:weakestSDPreg}), without changing
the optimizer. This block-diagonality, together with the constraint
that $\mathrm{diag}(X)=\mathbf{1}_{L}$, automatically implies $X^{(i,i)}=\mathbf{I}_{K^{(i)}}$,
as we have explained above. Since the von Neumann entropy splits additively
for block-diagonal arguments and since the cost matrix $C=-Q$ is
block-diagonal (cf. (\ref{eq:Qblockdiag})), it follows that the reindexed
diagonal block $\left[X_{\beta}\right]_{m}$ of $X_{\beta}$ is the
solution of the smaller problem 
\begin{align*}
\underset{Y\in\R^{L_{m}\times L_{m}}}{\text{minimize}}\ \  & -\mathbf{1}_{L_{m}}^{\top}Y\mathbf{1}_{L_{m}}+\beta^{-1}S(Y)\\
\text{subject to}\ \  & \mathrm{diag}(Y)=\mathbf{1}_{L_{m}},\\
 & Y\succeq0,
\end{align*}
 for each $m=1,\ldots,M$.

This problem is solved, via the Gibbs variational principle, as 
\[
Y=e^{\beta\,[\mathbf{1}_{L_{m}}\mathbf{1}_{L_{m}}^{\top}+\mathrm{diag}(\lambda)]},
\]
 where $\lambda\in L_{m}$ is a dual variable uniquely tuned to guarantee
$\mathrm{diag}(Y)=\mathbf{1}_{L_{m}}$.

It is not difficult to solve explicitly for $\lambda$. Consider the
ansatz $\lambda=\mu\mathbf{1}_{L_{m}}$, yielding 
\[
Y=e^{\beta\,[\mathbf{1}_{L_{m}}\mathbf{1}_{L_{m}}^{\top}+\mu\,\mathbf{I}_{L_{m}}]}=e^{\beta\mu}\,e^{\beta\,\mathbf{1}_{L_{m}}\mathbf{1}_{L_{m}}^{\top}}.
\]
 Now for normalized $u$ of general size, we can verify by Taylor
series expansion the identity 
\[
e^{\alpha uu^{\top}}=\mathbf{I}+(e^{\alpha}-1)uu^{\top}.
\]
 Applying in the case where $u=\mathbf{1}_{L_{m}}/\sqrt{L_{m}}$ and
$\alpha=\beta L_{m}$, we obtain 
\begin{equation}
Y=e^{\beta\mu}\left[\mathbf{I}_{L_{m}}+\frac{e^{\beta L_{m}}-1}{L_{m}}\,\mathbf{1}_{L_{m}}\mathbf{1}_{L_{m}}^{\top}\right].\label{eq:Ymu}
\end{equation}
 Then the diagonal entries of $Y$ are all 
\[
e^{\beta\mu}\left[1+\frac{e^{\beta L_{m}}-1}{L_{m}}\right].
\]
 We want to tune the dual variable $\mu$ so that this expression
equals $1$. Therefore we choose $\mu$ such that 
\[
e^{\beta\mu}=\frac{1}{1+\frac{e^{\beta L_{m}}-1}{L_{m}}}=\frac{L_{m}}{L_{m}+e^{\beta L_{m}}-1}.
\]

Then substituting into (\ref{eq:Ymu}), we obtain the optimizer $Y=[X_{\beta}]_{m}$:
\[
[X_{\beta}]_{m}=\tau_{\beta,m}\mathbf{I}_{L_{m}}+(1-\tau_{\beta,m})\mathbf{1}_{L_{m}}\mathbf{1}_{L_{m}}^{\top},
\]
 which is a convex combination of the identity matrix and the matrix
of all 1's, where the mixing coefficient $\tau_{\beta,m}\in(0,1)$
is given by 
\[
\tau_{\beta,m}=\frac{L_{m}}{L_{m}+e^{\beta L_{m}}-1}.
\]
 Note that 
\[
\lim_{\beta\ra\infty}\tau_{\beta,m}=0,
\]
 so 
\[
\lim_{\beta\ra\infty}\left[X_{\beta}\right]_{m}=\mathbf{1}_{L_{m}}\mathbf{1}_{L_{m}}^{\top}.
\]
 Since $X_{\beta}$ is block-diagonal (after reindexing), it follows
that $\lim_{\beta\ra\infty}X_{\beta}=Q$.
\end{proof}

\bibliographystyle{plain}
\bibliography{pps}

\end{document}